\documentclass[final]{siamart1116}

\usepackage{amsfonts,amsmath,amssymb,mathrsfs, graphicx}
\usepackage[square, sort, comma, numbers]{natbib}
\usepackage{algorithm, algorithmic}
\usepackage{array, url, multicol, enumerate,caption,bm,booktabs}
\usepackage{color}
\usepackage{lipsum}
\usepackage{epstopdf}
\usepackage{subfig}
\usepackage{tikz}

\newtheorem{problem}{Problem}[section]
\newtheorem*{remark}{Remark}

\newcommand{\V}{\mathbb{V}}
\newcommand{\Mat}{Mat\'{e}rn }
\newcommand{\bigO}{\mathcal{O}}

\newcommand{\W}{\mathcal{W}}
\newcommand{\R}{\mathbb{R}}

\newcommand{\E}{\mathbb{E}}

\newcommand{\bbf}{\mathbf{f}}
\newcommand{\bn}{\mathbf{n}}
\newcommand{\bq}{\mathbf{q}}
\newcommand{\bu}{\mathbf{u}}
\newcommand{\bv}{\mathbf{v}}

\newcommand{\x}{\mathbf{x}}
\newcommand{\y}{\mathbf{y}}

\newcommand{\KL}{Karhunen-Lo\`{e}ve\ }

\newcommand{\norm}[1]{\left\Vert#1\right\Vert}
\newcommand{\abs}[1]{\left\vert#1\right\vert}

\newcommand{\curly}[1]{\left\{#1\right\}}

\newcommand{\brac}[1]{\left[#1\right]}

\renewcommand{\div}{\operatorname{div}}
\hypersetup{
colorlinks=true,
linkcolor=blue,
citecolor=red}

\title{A Multilevel, Hierarchical Sampling Technique for Spatially Correlated
Random Fields
\footnote{T\lowercase{his work is performed under the auspices of the} U.S.
D\lowercase{epartment of} E\lowercase{nergy under} C\lowercase{ontract}
DE-AC52-07NA27344}}
\author{\normalsize Sarah Osborn\thanks{Center for Applied Scientific Computing, Lawrence Livermore National Laboratory, 7000 East Avenue, Livermore, CA 94551.} 
    ({\tt osborn9@llnl.gov}), ({\tt panayot@llnl.gov}). 
    \and Panayot S. Vassilevski\footnotemark[2] 
\and Umberto Villa\thanks{Institute for Computational Engineering and Sciences, University of Texas, Austin, TX. ({\tt uvilla@ices.utexas.edu}).}}

\begin{document}
\maketitle
\begin{abstract}
\noindent We propose an alternative method to generate samples of a spatially correlated random field with applications to large-scale problems for forward propagation of uncertainty. A classical approach for generating these samples is the \KL (KL) decomposition. However, the KL expansion requires solving a dense eigenvalue problem and is therefore computationally infeasible for large-scale problems. Sampling methods based on stochastic partial differential equations provide a highly scalable way to sample Gaussian fields, but the resulting parametrization is mesh dependent. We propose a multilevel decomposition of the stochastic field to allow for scalable, hierarchical sampling based on solving a mixed finite element formulation of a stochastic reaction-diffusion equation with a random, white noise source function. Numerical experiments are presented to demonstrate the scalability of the sampling method as well as numerical results of multilevel Monte Carlo simulations for a subsurface porous media flow application using the proposed sampling method. 
\end{abstract}

\begin{keywords}
multilevel methods, PDEs with random input data, mixed finite elements, 
uncertainty quantification, multilevel Monte Carlo
\end{keywords}

\begin{AMS}
65C05, 60H15, 35R60, 65N30, 65M75, 65C30
\end{AMS}

\section{Introduction}
\label{sec:intro}

Generating spatially correlated Gaussian random fields with specific
statistical properties is an important component and active topic of research in
a diverse array of application areas such as ecology, meteorology and geology
\cite{bolin2011spatial,cressie1993statistics,hu2013multivariate}. Gaussian
random fields are collections of random variables indexed by elements from a
multidimensional space with the property that their joint distribution is
Gaussian. These types of fields are specified by expectations and positive
semi-definite covariance functions and are generally good models for many phenomena;
see, e.g., \cite{abrahamsen1997review}.

In this paper we focus on geophysics
applications where the spatially correlated random field represents the
permeability of a porous medium. Specifically, we consider subsurface flow in a
specified domain, $D\subset \mathbb{R}^{d}$, governed by Darcy's law,
\begin{align}
{\bf q}({\bf x}) + k({\bf x})\nabla p({\bf x}) &= 0,
\label{eq:Darcy1} \\
\div {\bf q}({\bf x}) &= f({\bf x}),
\label{eq:Darcy2}
\end{align}
where ${\bf q}({\bf x})$ is the fluid velocity, $p({\bf x})$ is the fluid
pressure and $k({\bf x})$ the hydraulic conductivity, measuring the
transmissibility of the fluid through the porous medium.

In practice, only small portions of the hydraulic conductivity are known via
noisy data measurements and this contributes a large source of uncertainty into
the Darcy model. Understanding the effects that such uncertainties introduce
into the model, or forward propagation uncertainty quantification (UQ)
\cite{goodarzi2013introduction,smith2013uncertainty}, is quite important.  To
quantify this uncertainty we consider Monte Carlo methods, in particular
multilevel Monte Carlo methods (MLMC), where equations
\eqref{eq:Darcy1}-\eqref{eq:Darcy2} are solved repeatedly using different
realizations of the hydraulic conductivity field. In multilevel Monte Carlo methods,
more accurate (and expensive) simulations are run with fewer samples, 
while less accurate (and inexpensive) simulations are run with a larger number of samples. 
At the end of these simulations, quantities of interest depending on the  
velocity and/or pressure, such as the effective permeability, can be computed. 
Since a large number of Monte Carlo simulations are needed to produce accurate quantities of interest,
it is essential to have efficient algorithms to rapidly generate these
different realizations. 

A standard way of modeling the hydraulic conductivity is as a log-normal field,
$k({\bf x},\omega)= e^{\theta({\bf x},\omega)}$, where $\theta({\bf x},\omega)$ is a random field with a specified covariance
structure \cite{christakos2012modern,gelhar1993}. A classical, and well-studied
approach for generating these samples with the desired statistical properties
is the \KL decomposition. This sampling technique amounts to
solving an eigenvalue problem with the dense covariance matrix. The samples are
then different (random) linear combinations of the eigenvectors. However, sampling in
this manner suffers from a significant computational cost and high memory requirement. In fact, the
computational complexity for solving a dense eigenvalue problem grows cubically
with the size of the covariance matrix and the memory necessary to store all the eigenvectors in the \KL decomposition scales quadratically, making such methods prohibitively expensive for large-scale problems of interest. New approaches based on randomized methods and hierarchical semi-separable matrices can drastically reduce the cost of solving the eigenvalue problem (see e.g. \cite{SaibabaLeeKitanidis16}), however they allow to compute only the dominant eigenmodes of the \KL expansion and, therefore, introduce bias in the sampling.

A different approach for generating random field samples is by solving a
stochastic partial differential equation (SPDE) with a white noise source
function~\cite{whittle1954stationary,whittle1963stochastic,lindgren2011explicit}. To use the inverse of an elliptic differential operator as covariance function is a common approach for the solution of
large-scale Bayesian inverse problem governed by PDE forward models, see e.g. \cite{Stuart10,Bui-ThanhGhattasMartinEtAl13}, as it allows for efficient evaluation of the covariance operator using a fast and scalable multigrid solver.
The authors in \cite{lindgren2011explicit} provide a link between \KL
sampling from a Mat\'ern distribution and SPDE sampling, further motivating the SPDE approach. For example,
one can solve the following reaction-diffusion equation
\begin{equation}
-\Delta \theta ({\bf x},\omega) + \kappa^2 \theta ({\bf x},\omega)= \mathcal{W}({\bf
x},\omega)
\label{eq:ReactionDiffusion}
\end{equation}
where $\theta({\bf x},\omega)$ is the logarithm of the hydraulic conductivity,
$\kappa^2$ is a constant depending on the correlation length, and
$\mathcal{W}({\bf x},\omega)$ is a white noise function. This sampling approach
has two significant benefits: the first is avoiding the computational cost
incurred by solving a dense eigenvalue problem and the second is that optimal
solution methods for solving sparse linear systems arising from the finite
element discretization of \eqref{eq:ReactionDiffusion} can be applied. Despite
the benefits of this approach, it is not without imperfections as generated
samples contain artificial boundary effects.

Compared to the work in \cite{lindgren2011explicit}, three new ideas are introduced in this paper:
\emph{i}) a mixed discretization of the SPDE, \emph{ii}) a hierarchical version of the sampler, and \emph{iii}) mitigation of boundary artifacts using embedded domains.
Specifically, given a general unstructured fine grid, we construct a hierarchy of
algebraically coarsened grids and finite element spaces using the element-based algebraic multigrid techniques (AMGe) presented in  \cite{lashuk2012element,lashuk2014construction}, which provide coarse spaces with improved approximation properties than the ones from the original work \cite{pasciak2008exact}.
Then we generate random field samples at each level of the hierarchy by 
solving linear systems arising from the mixed discretization of \eqref{eq:ReactionDiffusion}. This allows us to
compute our (piecewise constant) samples in a hierarchical fashion (as needed in MLMC simulations), while leveraging existing scalable methods and software for
solving deterministic PDEs. 
We remark that while a hierarchical sampler can be constructed in a similar way for the primal formulation of the SPDE using geometric multigrid hierarchies, our framework offers more flexibility with respect to the geometry of the physical domain (since it does not require a sequence of nested grids). In addition, the computational advantages of resorting to the mixed formulation are twofold. 
First, the finite element discretization of the SPDE
\eqref{eq:ReactionDiffusion} requires the assembly of the square root of a mass matrix. In \cite{lindgren2011explicit}, the use of a continuous Galerkin (CG) finite element space in the discretization of the primal formulation leads to a non-diagonal mass matrix and mass lumping is used to make computation of the square-root tractable at the cost of accuracy.
On the contrary, the finite element pair of lowest Raviart-Thomas and piecewise constant functions in the mixed formulation leads to a diagonal mass matrix for the variable $\theta$. Second, in the AMGe framework, the construction of $L^2$ and $H(\div)$-conforming spaces (required in the mixed formulation) on agglomerated meshes is simpler than the one of $H^1$-conforming spaces (required by the primal formulation), and also leads to sparser, i.e. faster to apply, grid transfer operators.

Finally, we can not stress enough that performing realistic simulations with MLMC requires a considerable computational cost. To make the method
more feasible in practice, parallelism must be fully exploited. 
Our work focuses solely on parallelism across the spatial domain in computing realizations
of the input random field, then performing the subsequent solve of the model of interest. 
Further parallelism could be added using the scheduling approaches suggested in \cite{GmeinerDRSW16},
where the authors investigate the complex task 
of scheduling parallel tasks within and across levels of MLMC. 

The paper is structured as follows. In Section \ref{sec:gaussian-sampling} we give an
overview of both the classical KL expansion and the SPDE sampling approach.
The hierarchical SPDE sampling procedure is introduced and examined in Section \ref{sec:ml-sampling}.
Section \ref{sec:mlmc} contains a brief review of the
multilevel Monte Carlo method (MLMC), a 
scalable alternative to standard Monte Carlo methods that uses the solution of the PDE on a hierarchy of grid to effectively reduce the variance of the estimator. Here, we also present our
numerical results when our proposed sampling technique is used inside the MLMC
method. Lastly, Section \ref{sec:conclusions} contains our concluding remarks.
 
\section{Sampling from Gaussian Random Fields} 
\label{sec:gaussian-sampling}
In this section we review two ways of generating samples of a log-normal random
field. These are collections of random variables, $\{ \theta({\x},\omega):~{\x}\in D,~\omega\in \Omega \}$, where $\Omega$ is the sample space for the
probability space $(\Omega,\mathcal{F},\mathbb{P})$. For a fixed point, ${\bf
x}_{0}\in D$, $\theta({\x}_{0},\omega)$ is a random variable. For a fixed
$\omega_{0}\in \Omega$, $\theta({\x},\omega_{0})$ is a deterministic function
called a realization or sample of the random field.  

Random fields with a particular covariance functions are used to model the
correlation between points in spatial data. In geostatistical applications, the
logarithm of the hydraulic conductivity, \sloppy{${\theta({\x},\omega)=\log{(k({\bf x},\omega))}}$} is modeled as a random field with a
specific class of covariance function. 

In particular, we consider a stationary isotropic Gaussian field with the widely used class
of \Mat covariance functions~\cite{matern1986spatial}, given by
\begin{equation}
    \operatorname{cov}(\x, \y) = \frac{\sigma^{2}}{2^{\nu -1}\Gamma(\nu)} (\kappa \Vert
    {\bf y}-{\bf x}\Vert)^{\nu}K_{\nu}(\kappa \Vert {\bf y}-{\bf x}\Vert),
    \label{eq:Matern}
\end{equation}
where $\sigma^{2}$ is the marginal variance, $\nu>0$ determines the mean-square
differentiability of the underlying process, $\kappa>0$ is a scaling factor
inversely proportional to the correlation length, $\Gamma(\nu)$ is the gamma function, 
and $K_{\nu}$ is the modified
Bessel function of the second kind. When 
$\nu=\frac{1}{2}$ \eqref{eq:Matern} reduces to the common exponential covariance
function, given by 
\begin{equation}\label{eq:exp_cov_func}
    \operatorname{cov}(\x,\y) = \sigma^2 e ^{-\kappa \norm{\y - \x}} .
\end{equation}
In this section we consider two ways
for generating such realizations or samples of random fields. The first method
we consider is the classical \KL expansion. The second is based
on solving a particular stochastic reaction-diffusion equation with a
white-noise source function.

\subsection{The \KL Expansion}
The \KL (KL) expansion of a second-order random field \footnote{ A random field
    $\theta(\x,\omega)$ is \emph{second-order} if
for each $\x \in D$ the random
variable $\theta(\x)$ has finite variance.} provides a series representation using the orthonormal basis provided by the eigenfunctions of the underlying covariance operator~\cite{loeve1978probability}.
For a bounded regular domain $D$ we define the convolution operator
            $$\mathcal{C}v(\x)=\int_{D}\operatorname{cov}(\x,\x')v(\x')d\x'.$$
Since $\mathcal{C}$ is a compact operator, the eigenvalue problem
\begin{equation}\label{eq:ew}
    (q,\mathcal{C}v)_{D}=\lambda(q,v)_{D}\quad \forall q\in L^2(D), 
\end{equation}
            admits a countable sequence of eigenpairs $(\lambda_{i},v_{i})$
            where $\lim_{i\rightarrow\infty}\lambda_{i}=0.$ 
A Gaussian random field can be expanded as
          $$\theta(\x,\omega)=\sum_{i=0}^{\infty}\xi_{i}(\omega)\sqrt{\lambda_{i}}v_{i}(\x),\quad \xi_{i}(\omega)\sim N(0,\sigma^{2})\;\text{i.i.d.} $$

          In practice, a discrete version of the eigenvalue problem \eqref{eq:Matern} is computed. A triangulation of the domain $\mathcal{T}_h \subset D$ is generated with the discrete function space $\Theta_h \subset L^2(D)$ of piecewise constant functions. Then, we can compute realizations of the Gaussian field with a \KL expansion truncated after $M_h$ terms as 
          \begin{equation}
              \theta_{h}(\omega)=\sum_{i=0}^{M_{h}}\xi_{i}(\omega)\sqrt{\lambda_{h,i}}v_{h,i}(\x),\quad \xi_{i}(\omega)\sim N(0,\sigma^{2})\;\text{i.i.d}, 
      \end{equation}
      where the  pairs $(\lambda_{i},v_{i})$ solve the generalized eigenvalue
problem:
            $$(q_{h},\mathcal{C}_{h}v_{h,i})_{h}=\lambda_{h,i}(q_{h},v_{h,i})_{h}$$
        with inner product defined as $(r,s)_h=\int_{\mathcal{T}_h}r\,s \, d\x$ for $r,s \in \Theta_h.$ 

One of the main reasons for sampling using this approach is that any truncation
of the KL expansion gives a sample with the minimal mean square error, see e.g.
\cite{keating1983note}. However, this method is not computationally feasible
due to the high computational cost for solving a dense eigenvalue problem. On a
mesh with $N$ degrees of freedom, the cost to factor a dense covariance matrix
is $\mathcal{O}(N^{3})$. This makes generating samples in this manner
impractical for finely resolved meshes where the number of degrees of freedom
could be of the order of millions. This motivates other sampling methods that
do not suffer from such a complexity cost.

\subsection{Stochastic PDE Sampling}
\label{sec:PDESampling}

An important link between Gaussian fields and Gaussian Markov random fields is established in \cite{lindgren2011explicit}, where a random process on $\R^d$ with a \Mat covariance function can be obtained as the solution of a particular stochastic partial differential equation (SPDE). This provides an alternative method for computing the desired samples of a Gaussian random field via a SPDE, as opposed to the computationally intensive KL expansion. 

Realizations of a Gaussian random field with an underlying \Mat covariance, $\theta({\bf x},\omega)$,  solve the following linear stochastic PDE:
\begin{equation}
(\kappa^{2} - \Delta)^{\alpha/2}\theta(\x,\omega) = g \W({\bf x},\omega) \quad  \x\in \R^{d} , \ \alpha = \nu +\frac{d}{2}, \ \kappa> 0, \nu >0,
\label{eq:Fractional}
\end{equation}
for some other realization of the standard Gaussian white noise function $\W$
with scaling factor $g$, \cite{whittle1954stationary,whittle1963stochastic}. The scaling factor is chosen to be
\begin{equation*}
    g = (4\pi)^{d/4} \kappa^{\nu}  \sqrt{\frac{\Gamma\left(\nu + d/2\right)}{\Gamma(\nu)}},
\end{equation*}
to impose unit marginal variance of $\theta(\x,\omega)$~\cite{bolin2011spatial}. 

It is worth noting that if $\nu=1$ in two dimensions, and $\nu=\frac{1}{2}$ in three dimensions, then \eqref{eq:Fractional} reduces to the following standard reaction-diffusion equation,

\begin{equation}
    (\kappa^{2}-\Delta)\theta(\x, \omega) = g\mathcal{W}(\x,\omega).
\label{eq:SPDE}
\end{equation}
Additionally in three dimensions, realizations of a Gaussian random field with
exponential covariance are solutions of \eqref{eq:SPDE}.
Thus a scalable sampling alternative is equivalent to efficiently solving 
the stochastic reaction-diffusion equation given by
\eqref{eq:SPDE} where we are able to leverage existing scalable solution strategies.  
We remark that, as explained in \cite{lindgren2011explicit}, samples from \Mat
distributions with $\nu = 2k + 1$ (in two spatial dimensions) and from \Mat
distributions with $\nu = 2k + \frac{1}{2}$ (in three spatial dimensions), can be obtained by recursively solving $(\kappa^{2}-\Delta)\theta_{i+1}(\x, \omega) = \theta_{i}(\x, \omega)$ for $i = 1, \ldots, k$, where $\theta_{0}(\x, \omega)$ is the solution of \eqref{eq:SPDE}.

Equation \eqref{eq:SPDE} is discretized using a mixed finite element method, \cite{boffi2013mixed,fortin1991mixed}.

Following standard notation, for scalar functions $\theta, q\in L^2(D)$ and vector functions $\bu, \mathbf{v} \in \mathbf{L}^2(D)=[L^2(D)]^d$, we define the inner products:
$$ (\theta,q) = \int_D \theta\, q \;d\x \mbox{   and    } (\bu, \mathbf{v})=\int_D \bu \cdot \mathbf{v}\; d\x.$$
We also define the functional spaces $\mathbf{R}$ and $\Theta$ as 
$$
\begin{array}{lcr}
    \mathbf{R} = H(\div;D) :=\curly{\bu \in \mathbf{L}^2(D) \mid \operatorname{div}\, \bu \in L^2(D) \text{ and } \bu \cdot \mathbf{n} = 0 \text{ on }\partial D}\text{ and } \\ \Theta = L^2(D).
\end{array}
$$
Finally, we introduce the bilinear forms
\begin{equation*}
\begin{array}{lll}
m(\bu, \bv)  &:=(\bu, \bv)       & \forall \, \bu, \bv \in \mathbf{R},\\
w(\theta, q) &:= (\theta, q)     & \forall \, \theta, q \in \Theta,\\
b(\bu, q)    &:= (\div\, \bu, q) & \forall \, \bu \in \mathbf{R}, q \in \Theta,\\
\end{array}
\end{equation*}
and the linear form
\begin{equation*}
F^\mathcal{W}(q) := (\mathcal{W}, q) \quad \forall \, q \in \Theta.
\end{equation*}

Following standard finite element techniques, let $\mathbf{R}_h \subset \mathbf{R}$ denote the lowest order Raviart-Thomas finite element space and $\Theta_h \subset \Theta$ denote the finite element space of piecewise constant functions. 
Then, a mixed finite element discretization of \eqref{eq:SPDE} reads
\begin{problem}
Find $(\bu_h, \theta_h)\in \mathbf{R}_h \times \Theta_h$ such that
    \begin{equation}
       \begin{array}{ll}
            m(\bu_h, \bv_h) + b(\bv_h, \theta_h) = 0 & \forall \bv_h \in \mathbf{R}_h \\
            b(\bu_h, q_h) - \kappa^2\, w(\theta_h, q_h) = -g\, F^\mathcal{W}(q_h) & \forall q_h \in \Theta_h.
        \end{array}\label{eq:spde_weak}
    \end{equation}
    with essential boundary conditions $\bu_h \cdot \mathbf{n}=0$.
    \label{prob:spde_weak}
\end{problem}
\begin{remark}
The choice of a low order finite element discretization is optimal with respect to the regularity of the solution. For example, in 3D space and for $\nu = \frac{1}{2}$ the realizations of a Gaussian random field with \Mat covariance are only \emph{almost surely} H\"older continuous with any exponent $\beta < \frac{1}{2}$, see e.g. \cite{charrier2012strong}.
\end{remark}

In the following, we denote the discrete linear algebra representations of the bilinear forms $m$, $b$, and $w$ with the matrices $M_h$, $W_h$, and $B_h$ where $M_h$ is the mass matrix for the space $\mathbf{R}_h$, $W_h$ is the mass matrix for the space $\Theta_h$, and $B_h$ stems from the divergence operator. Particular care must be taken for the
linear algebra representation of the stochastic right hand side $F^\mathcal{W}(q_h)$. This
requires recalling the following two properties of Gaussian white noise defined
on a domain $D$. For any set of test functions $\curly{q_i \in L^2(D), i=1,\dots,n},$
the expectation and covariance measures are given by
\begin{align}
    \mathbb{E}[ (q_{i},\mathcal{W})] &= 0, \\
    \operatorname{cov}\left( (q_{i},\mathcal{W}), (q_{j}, \mathcal{W})\right)
    &= (q_{i},q_{j}).
\end{align}
By taking $q_{i},q_{j}$ as piecewise constants so that $q_{i},q_{j}\in \Theta_h$, the
second equation implies that the covariance measure over a region of the domain
is equal to the area of that region \cite{lindgren2011explicit}.
 
As a result of these properties, the computation of the discrete stochastic linear
functional amounts to computing
\begin{equation*}
    f_h = W_h^{\frac{1}{2}}\xi_h(\omega),
\end{equation*}
where the coefficients of $\xi_h(\omega)$ in the finite element expansion form a random vector drawn from $\mathcal{N}(0,I)$.
It should be noted that the mass matrix $W_h$ for the space $\Theta_h$ is diagonal, hence its square root can be computed cheaply.
This is not the case in the original primal formulation presented in \cite{lindgren2011explicit} since they employ piecewise linear continuous elements. The authors suggest then using mass lumping to make computation feasible at the cost of accuracy.

Then the discrete mixed finite element problem can be written as the linear system,  
\begin{equation} 
    \mathcal{A}_hU_h = 
    \begin{array}{lcr}
        \begin{bmatrix} 
            M_h & B_h^T \\ 
        B_h & -\kappa^2 W_h \end{bmatrix} 
            \begin{bmatrix} \bu_h  \\ 
                \theta_h 
            \end{bmatrix} 
            = \begin{bmatrix} 0  \\ -g\,f_h(\omega)
            \end{bmatrix} = F_h,
            
        \end{array} \label{eq:linsys}
    \end{equation}
where $f_h(\omega) \sim \mathcal{N}(0,W_h).$
An efficient iterative solution method of \eqref{eq:linsys}, required to scalably generate the desired Gaussian field realizations, is described in the following section.

It is worth noticing that the lowest order Raviart-Thomas spaces allow us to define a discrete gradient
$\nabla_h: \Theta_h \mapsto \mathcal{R}_h$ such that the identity $w(\nabla_h \theta_h, \bv_h) = -b(\bv_h, \theta_h)$ holds for all $\theta_h \in \Theta_h$ and all $\bv_h \in \mathcal{R}_h$ with zero normal trace on $\partial D$. Specifically, we define $\nabla_h = -M_h^{-1}B^T_h$.
Then, the Schur complement of \eqref{eq:linsys} with respect to $\theta_h$ can be viewed as a non-local
{\em discontinuous} Galerkin (interior penalty) discretization of the original PDE \eqref{eq:SPDE}, cf., \cite{RVW}. Namely,  we have
\begin{equation*}
m(\nabla_h \theta_h,\; \nabla_h q_h) + \kappa^2\, w(\theta_h,\;q_h) = g F^W(q_h), \quad \forall\, q_h \in \Theta_h.
\end{equation*}
The above formulation allows us to apply the theory in
\cite{lindgren2011explicit} to show convergence of $\theta_h$ to $\theta$,
however a detailed analysis of the convergence of the mixed system
\eqref{eq:spde_weak} is outside the scope of this work. 

\subsection{SPDE Sampler Numerical Solution}\label{sec:spde_solver}
The sparse large matrix $\mathcal{A}_h$ admits the following block UL decomposition
\begin{equation}
\mathcal{A}_h = \begin{bmatrix}A_h & B_h^T \\ 0 & -\kappa^2 W_h \end{bmatrix}\begin{bmatrix} I & 0 \\ -\kappa^{-2}W_h^{-1}B_h & I \end{bmatrix},
\end{equation}
where the blocks $A_h = M_h + \kappa^{-2}B_h^T W_h^{-1}B_h$ and $-\kappa^{-2}W_h^{-1}B_h$ are both sparse since $W_h$ is diagonal.
Then solving the large sparse linear system ~\eqref{eq:linsys} amounts to first
finding $\bu_h$ such that
\begin{equation}\label{eq:schur}
 A_h \bu_h = - g\kappa^{-2} B_h^T f_h(\omega),
\end{equation}
and then to set
\begin{equation}
\theta_h = \kappa^{-2}W_h^{-1}\left( B_h \bu_h + g\,f_h(\omega) \right).
\end{equation}
It is worth noting that $A_h$ in \eqref{eq:schur} is symmetric positive definite and stems from the matrix representation of the weighted $H(\div)$ inner product,
$$ a(\bu_h, \bv_h) := (\bu_h, \bv_h) + \kappa^{-2}( \div\, \bu_h, \div \, \bv_h ),$$
thanks to the particular choice of the spaces $\mathbf{R}_h$ and $\Theta_h$, see e.g. \cite{arnold2006finite,arnold2010finite}.
Then to solve the linear system \eqref{eq:schur}, we use the conjugate gradient (CG) method preconditioned by the Auxiliary Space AMG preconditioner for $H(\operatorname{div})$ problems in~\cite{kolev2012parallel}, which ensures a mesh independent convergence and robustness with respect to the choice of the correlation length $\kappa$. In particular, for our numerical results, we use hypre's HypreADS preconditioner~\cite{hypre}.

This approach allows for scalable sampling, but the parametrization of
$\theta_h(\omega)$ is mesh-dependent.
This is not suitable for MLMC where the same realization of the Gaussian field
needs to be computed at a fine and a coarse spatial resolution. In Section \ref{sec:ml-sampling}, we detail our 
proposed method that allows for hierarchical sampling based on a multilevel decomposition of the stochastic field.

\subsection{Boundary artifacts and embedded domains}
\label{sec:boundaryArtifacts}
As noted in \cite{lindgren2011explicit}, the SPDE sampling method introduces errors on the boundary resulting in a larger marginal variance of the random field close to the boundary of the domain. These errors are due
  to the introduction of artificial boundary conditions where the SPDE is
  defined globally, yet must be discretized on a finite domain. To mitigate this issue we embed the original mesh into a larger mesh. Equation \eqref{eq:Fractional}
is discretized on the larger domain $\overline D$ where a corresponding linear system of the form 
\eqref{eq:linsys} is solved 
and the corresponding random field realization is restricted to the original domain $D$. A rule of thumb proposed in \cite{lindgren2015bayesian} suggests the boundary effect is negligible at a distance equal to the correlation length from the boundary. Figure \ref{fig:mar_var} shows the sample marginal variance of 2000 Gaussian field samples with correlation length $0.01$, marginal variance $\sigma^2=1$ generated by the SPDE sampler. In Figure~\ref{subfig:orig}, the samples were computed using the original circular mesh with diameter equal $0.2$, whereas in Figure~\ref{subfig:embed} the SPDE given by \eqref{eq:Fractional}
is solved on the original mesh embedded in a larger square domain with sides of length $0.3$ shown in Figure~\ref{subfig:embed_mesh}. Thus, mesh embedding alleviates the issue of the error on the boundary.

\begin{figure}[htbp]
     \centering
     \begin{tabular}{c c c c}
         \subfloat[Sample Computed with Original Mesh]{\includegraphics[scale=.14]{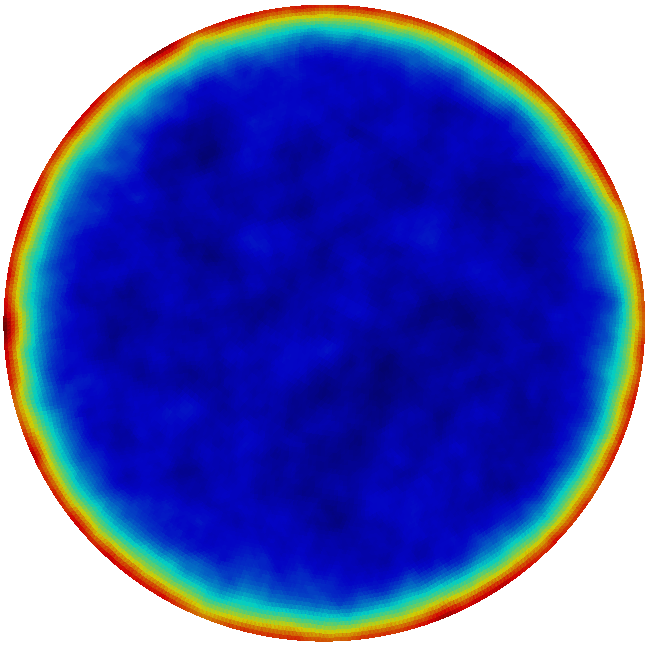}\label{subfig:orig}}
         &
         \subfloat[Original Mesh Embedded in Square]{\includegraphics[scale=.21]{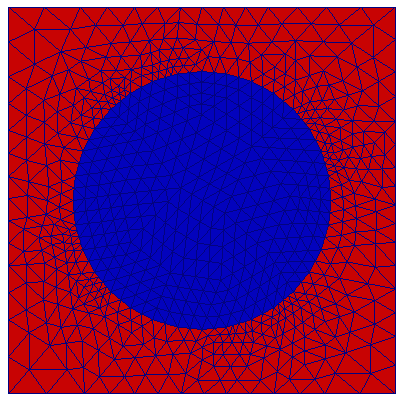}\label{subfig:embed_mesh}}
         &
         \subfloat[Samples Computed with Mesh Embedding]{\includegraphics[scale=.14]{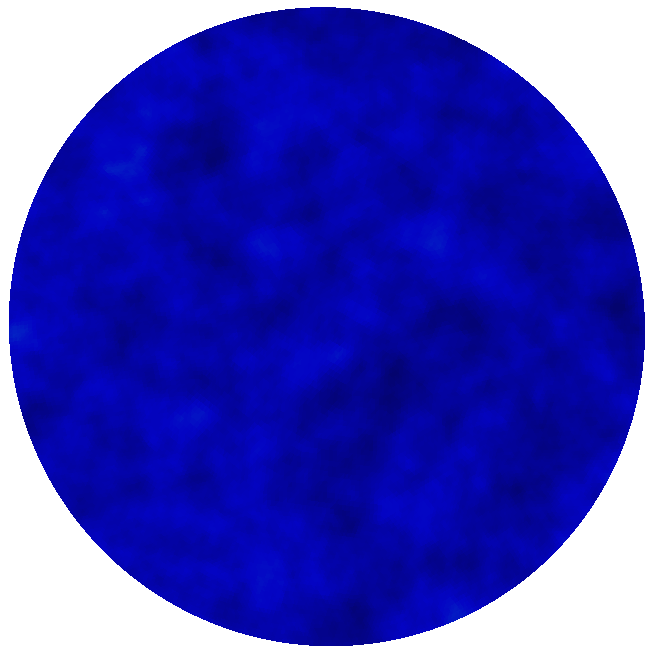}\label{subfig:embed}}
         &
         \subfloat{\includegraphics[scale=.19]{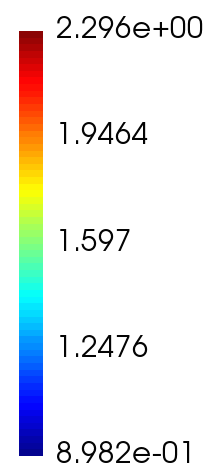}}

     \\\end{tabular}
     \caption{\small Sample marginal variance with 2000 samples of a Gaussian field with \Mat covariance function with correlation length $0.01$, marginal variance $\sigma^2=1$ computed using the SPDE sampling technique. (a) shows the marginal variance with the original circular mesh with diameter $0.2$. (b) shows the original circular mesh embedded in a square with sides of length $0.3$, and (c) shows the marginal variance when the mesh embedding procedure is used, noting that the boundary effect is negligible at a distance equal to the correlation length from the boundary. The error on the boundary from the artificial boundary conditions is mitigated when the domain is embedded in a larger mesh.} 
    \label{fig:mar_var}
 \end{figure}
 
\section{Multilevel, Hierarchical Sampling}
\label{sec:ml-sampling}

In this section we describe our proposed hierarchical sampling technique based on a multilevel decomposition of the stochastic field for sampling Gaussian fields based on solving the SPDE given by~\eqref{eq:SPDE}. We first briefly describe the agglomeration process of a fine grid into a sequence of coarser levels, introducing the necessary finite element spaces and inter-level operators that will be used. After introducing the multilevel structure of the stochastic field, the implementation details of the proposed method are described, followed by numerical results demonstrated the scalability of the method. 

\subsection{Multilevel Structure}
Using methodology from element-based algebraic multigrid (AMGe), we are able to construct operator-dependent coarse spaces with guaranteed approximation properties on general, unstructured grids, see ~\cite{lashuk2014construction,lashuk2012element,pasciak2008exact} for further details.

Let $\mathcal{T}_{0}$ denote a fine grid discretization of the domain $D$. This fine grid is agglomerated into a hierarchy of $L$ coarser algebraic levels, $\{\mathcal{T}_{\ell}\}_{\ell=1}^{L}$ where $L$ denotes the coarsest grid. Agglomerates are formed by grouping together fine-grid elements.
Based upon this hierarchy, we have the sequence of spaces $\mathbf{R}_{\ell},\Theta_{\ell}$ for $\ell=0,\dots, L$, that are the discrete analogues for $H(\operatorname{div},D)$ and $L^{2}(D)$ respectively on the discretizations $\{\mathcal{T}_{\ell}\}_{\ell=0}^{L}$. The space $\mathbf{R}_{0}$ is discretized by lowest order Raviart-Thomas finite elements and $\Theta_{0}$ is discretized by piecewise constant finite elements.

In addition, we will define the following interpolation operators from the coarse space $\Theta_{\ell+1}$ to fine space $\Theta_{\ell}$ where,
\begin{equation}
P_{\theta}: \Theta_{\ell+1}\rightarrow \Theta_{\ell} \quad \text{for} \ \ell=0,\ldots, L-1.
\end{equation}
We also define the operators from the coarse space $\mathbf{R}_{\ell+1}$ to the fine space $\mathbf{R}_{\ell}$ as
\begin{equation}
    P_{\mathbf{u}}: \mathbf{R}_{\ell+1}\rightarrow \mathbf{R}_{\ell} \quad \text{for} \ \ell=0,\ldots, L-1.
\end{equation}
Details on the construction and properties of the interpolation operators $P_{\theta}$ and $P_{\mathbf{u}}$ using AMGe techniques are given in \cite{pasciak2008exact, lashuk2012element, lashuk2014construction}. 
Here we limit our discussion to observe that such operators reduce to the canonical interpolation operators of geometric multigrid when a nested hierarchy of uniformly refined meshes is given (in our case of constant PDE coefficients).
In addition, we define  
$$ \mathcal{P} = \begin{bmatrix} P_{\mathbf{u}} & 0 \\ 0 & P_{\theta} \end{bmatrix}, $$
as the block interpolation operator for the matrix $\mathcal{A}_{\ell}.$ 
Using this hierarchical notation, a discrete realization of a Gaussian random field on level $\ell$ corresponding to the $\mathcal{T}_{\ell}$ discretization will be denoted as $\theta_{\ell}$. 
\subsection{Hierarchical SPDE Sampler}

We now propose a multilevel, decomposition of the random field and show 
that the coarse representation is a realization of the Gaussian random 
field on the coarse space $\Theta_{\ell+1}.$

\begin{proposition}
    The Gaussian random field $\theta_{\ell}(\omega)$ given by
$$ \begin{bmatrix} \mathbf{u}_{\ell} \\ \theta_{\ell}(\omega) \end{bmatrix} = \mathcal{A}^{-1}_{\ell} \begin{bmatrix} 0 \\ -gW^{1/2}_{\ell}\xi_{\ell}(\omega)\ \end{bmatrix}$$
    admits the following two-level decomposition:
    \begin{equation}\label{eq:bileveld}
        \theta_{\ell}(\omega)=P_{\theta}\theta_{\ell+1}(\omega)+\delta\theta_{\ell}(\omega),
    \end{equation}
    where $\theta_{\ell+1}(\omega)$ is a coarse representation of a Gaussian random field from
    the same distribution, and 
    \begin{equation}
    \begin{bmatrix} 
        \mathcal{A}_{\ell} & \mathcal{A}_{\ell}\mathcal{P} \\ 
        \mathcal{P}^T\mathcal{A}_{\ell} & 0 
    \end{bmatrix} 
    \begin{bmatrix} 
        \delta U_{\ell}  \\ 
        U_{\ell+1}
    \end{bmatrix} 
    = \begin{bmatrix} 
        \mathcal{F}_{\ell}  \\ 0
      \end{bmatrix},
      \label{eq:ml_sampling}
\end{equation}
with the block expressions given by 
$$ \delta U_{\ell} = \begin{bmatrix} \delta \mathbf{u}_{\ell} \\ \delta \theta_{\ell}(\omega) \end{bmatrix}, 
\; U_{\ell+1}=  \begin{bmatrix} \mathbf{u}_{\ell+1} \\  \theta_{\ell+1}(\omega) \end{bmatrix}, \mbox{ and }
\mathcal{F}_{\ell}= \begin{bmatrix} 0 \\ -g W_{\ell}^{1/2}\xi_{\ell}(\omega) \end{bmatrix}. $$
\end{proposition}
\begin{proof}
We start with a block LU factorization of the operator in \eqref{eq:ml_sampling}
\begin{equation*}
\begin{bmatrix} 
        \mathcal{A}_{\ell} & \mathcal{A}_{\ell}\mathcal{P} \\ 
        \mathcal{P}^T\mathcal{A}_{\ell} & 0 
\end{bmatrix}
\begin{bmatrix} 
        \delta U_{\ell}  \\ 
        U_{\ell+1}
\end{bmatrix} 
 =
\begin{bmatrix} 
    \mathcal{A}_{\ell} & 0 \\ 
     \mathcal{P}^T\mathcal{A}_{\ell} & -\mathcal{S} 
\end{bmatrix}
\begin{bmatrix} 
        I_{\ell} & \mathcal{P} \\ 
        0 & I_{\ell+1}
\end{bmatrix}
 \begin{bmatrix} 
        \delta U_{\ell}  \\ 
        U_{\ell+1}
\end{bmatrix}
=
\begin{bmatrix} 
        \mathcal{F}_{\ell}  \\ 0
\end{bmatrix}
,
\end{equation*}
where $I_{\ell}$ and $I_{\ell+1}$ are the identity matrices at level $\ell$ and $\ell+1$.
Here the Schur complement operator $\mathcal{S}$ is defined as
$$\mathcal{S} := (\mathcal{P}^T\mathcal{A}_{\ell}) \mathcal{A}_{\ell}^{-1} (\mathcal{A}_{\ell}\mathcal{P}) = \mathcal{P}^{T}\mathcal{A}_{\ell}\mathcal{P} = \mathcal{A}_{\ell+1},$$
where, by definition of the prolongation operator, $\mathcal{A}_{\ell+1}$ is the coarse operator stemming from the discretization of the SPDE operator at level $\ell + 1$.
Recalling the definition of $U_{\ell}$, the above LU factorization implies
\begin{equation}\label{eq:lu-twolevel}
\begin{array}{ll}
\mathcal{A}_{\ell} U_{\ell} &= \mathcal{F}_\ell\\
\mathcal{A}_{\ell+1} U_{\ell+1} &= \mathcal{P}^T\mathcal{A}_{\ell} U_{\ell} = \mathcal{P}^T \mathcal{F}_\ell\\
\delta U_{\ell} + \mathcal{P}U_{\ell+1} &= U_{\ell}.
\end{array},
\end{equation}

To show that $\theta_{\ell+1}(\omega)$ is a realization of a Gaussian random field
 on the coarse space $\Theta_{\ell+1},$ 
we observe that $\theta_{\ell+1}(\omega)$ satisfies
$$ \mathcal{A}_{\ell+1} \begin{bmatrix} {\mathbf u}_{\ell+1} \\ \theta_{\ell+1}(\omega) \end{bmatrix}
    = \mathcal{P}^T \begin{bmatrix} 0 \\ -gW_{\ell}^{1/2} \xi_{\ell}(\omega) \end{bmatrix} = \begin{bmatrix} 0 \\ -gW_{\ell+1}^{1/2} \xi_{\ell+1}(\omega) \end{bmatrix} $$
where $W_{\ell+1}=P_{\theta}^{T}W_{\ell}P_{\theta}$ is the mass matrix on the coarser level $\ell + 1$.
The random forcing term is defined as $\xi_{\ell+1}(\omega):=W_{\ell+1}^{-1/2}P_{\theta}^{T}W_{\ell}^{1/2}\xi_{\ell}(\omega)$. We note that $\xi_{\ell+1}(\omega)$ is a multivariate Gaussian vector with zero mean and covariance matrix
\begin{align*}
\operatorname{cov}( \xi_{\ell+1}(\omega) ) & = \mathbb{E}[\xi_{\ell+1}(\omega) \xi_{\ell+1}(\omega)^T ] \\
                                           & = 
    (W_{\ell+1}^{-1/2}P_{\theta}^{T}W_{\ell}^{1/2}) \mathbb{E}[\xi_{\ell}(\omega) \xi_{\ell}(\omega)^T](W_{\ell+1}^{-1/2}P_{\theta}^{T}W_{\ell}^{1/2})^T \\
    {} & = W_{\ell+1}^{-1/2}P_{\theta}^{T}W_{\ell}P_{\theta}W_{\ell+1}^{-1/2} = I,
\end{align*}
where we have exploited the fact that $\xi_{\ell}(\omega) \sim \mathcal{N}(0,I)$.

The two-level decomposition \eqref{eq:bileveld} then follows from \eqref{eq:lu-twolevel}$_3$ by noticing that $\theta_{\ell+1}(\omega)$ is the solution of the SPDE \eqref{eq:SPDE} discretized at level $\ell+1$ with random forcing term
$\xi_{\ell+1}(\omega)\sim\mathcal{N}(0,I)$.
\end{proof}

Thus given $\xi_{\ell}(\omega)\sim\mathcal{N}(0,I)$ we are able to efficiently sample the
Gaussian random field on both the fine and upscaled discretization using the SPDE sampler.

\subsection{Hierarchical SPDE Sampler Numerical Solution}
We now describe the solution procedure that we employ for the hierarchical SPDE sampler.
Starting with an unstructured mesh, a coarse problem is constructed by grouping 
together fine-grid elements using the graph partitioner METIS~\cite{metis} and  
the coarse finite element spaces are computed as described in detail in 
\cite{lashuk2012element, lashuk2014construction}. The same procedure is applied 
recursively so that a nested hierarchy of agglomerated meshes and coarse spaces is constructed.  
We will present results on general unstructured meshes, however the technique immediately 
translates to the case of nested refined meshes generated with uniform refinement 
by choosing the canonical interpolators. 

Given $\xi_{\ell}(\omega)$, we compute the realizations of the Gaussian field at levels $\ell$ and $\ell+1$ as follows. 
First we compute the sample for level $\ell+1$ by solving the saddle point system
$$\mathcal{A}_{\ell+1} U_{\ell+1}=\mathcal{P}^T \begin{bmatrix} 0 \\ -gW_{\ell}^{1/2} \xi_{\ell}\end{bmatrix},$$
using the methodology described in Section \ref{sec:spde_solver}.
Then we compute the sample for level $\ell$ by iteratively solving
\eqref{eq:linsys} with $\mathcal{P} U_{\ell+1}$ as the initial guess.

We conclude this section with a remark on the independence of different realizations of the random field 
for a particular spatial resolution generated by the hierarchical SPDE sampler.
Specifically, two realizations $\theta_{\ell}(\omega_i)$ and $\theta_{\ell}(\omega_j)$ are independent if and only 
if $\xi_{\ell}(\omega_i)$ and $\xi_{\ell}(\omega_j)$ are independent. This is of extreme importance for MLMC and  
requires the use of a quality random number generator. In our numerical experiments, 
we use Tina's Random Number Generator Library~\cite{trng}  
which is a pseudo-random number generator with dedicated support for 
parallel, distributed environments~\cite{bauke2007random}.

\subsection{Numerical Results}
The numerical results in this section are used to demonstrate the effectiveness
of our SPDE sampling approach. 
For the simulations, an absolute stopping criteria of $10^{-12}$ and a relative stopping 
criteria of $10^{-6}$ is used for the linear solver.
For the results presented in this and the following sections, we use the C++ finite element library MFEM~\cite{mfem} to assemble the discretized problems.

First we consider the performance of the sampler in two space dimensions. The physical domain 
$D=(0,1200)\times(0,2200)$ is embedded in a larger domain
$\overline D = (-100, 1300) \times (-100, 2300)$ to decrease the effect of the artificial Neumann boundary condition, see Section \ref{sec:boundaryArtifacts}. The computational domain is then discretized using a structured quadrilateral mesh with 294,400 elements in the physical domain. 
Then a hierarchy of unstructured agglomerated meshes is constructed using the graph partitioner METIS~\cite{metis} with a coarsening ratio of 8 elements per agglomerate, see Figure \ref{fig:2DSPE10MeshHierarchy}. Note that, on coarse levels, agglomerated elements have irregular shapes and an arbitrary number of neighboring elements. Coarse faces are also not flat.
Figure \ref{fig:2DSPE10Samples} shows a sequence of computed Gaussian random field samples $\theta(\x,\omega)$ on different levels with correlation length $b=100$.  Observe the similarity in the fields generated on the algebraically coarsened levels. Locations of the essential features of the fine grid sample are preserved, however contours are blurred and boundaries are spread out due to the irregular shape of the agglomerated elements.

We now investigate the performance of the hierarchical sampler under weak scaling, i.e. when the number of mesh elements is proportional to the number of processes. Structured coarsening is used to create the hierarchy of agglomerated meshes with a coarsening ratio of 4 elements per agglomerate, that is, the original mesh is uniformly refined to build the hierarchy of levels. 
The code was executed on \textit{Sierra} at Lawrence Livermore National Laboratory consisting of a total of 1,944 nodes where each node has two 6-core Xeon EP X5660 Intel CPUs (2.8 Ghz), and 24GB of memory. We use 8 MPI processes per node. Figure \ref{fig:2Dscaling} shows the number of MPI processes versus the average time to generate a realization of the Gaussian field for 1000 samples. The number of MPI processes ranges from 8 to 512 and size of the stochastic dimension on the fine grid ranges from 3.7 to 235 million. The proposed sampling method exhibits the optimal near linear scaling as the number of processes increases for all levels in the hierarchy.

 \begin{figure}[htbp]
     \centering
     \begin{tabular}{c c c c}
         \subfloat[Level $\ell=0$]{\includegraphics[scale=.14]{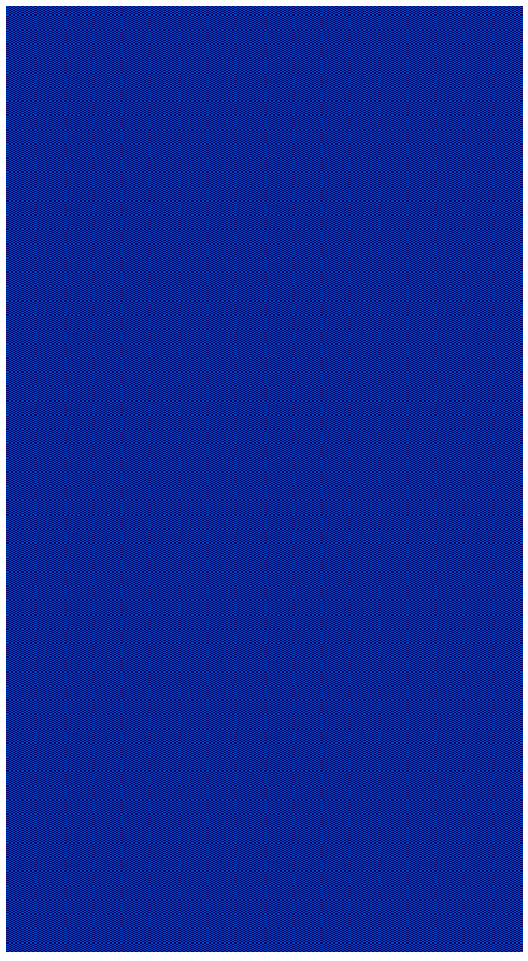}}
         &
         \subfloat[Level $\ell=1$]{\includegraphics[scale=.14]{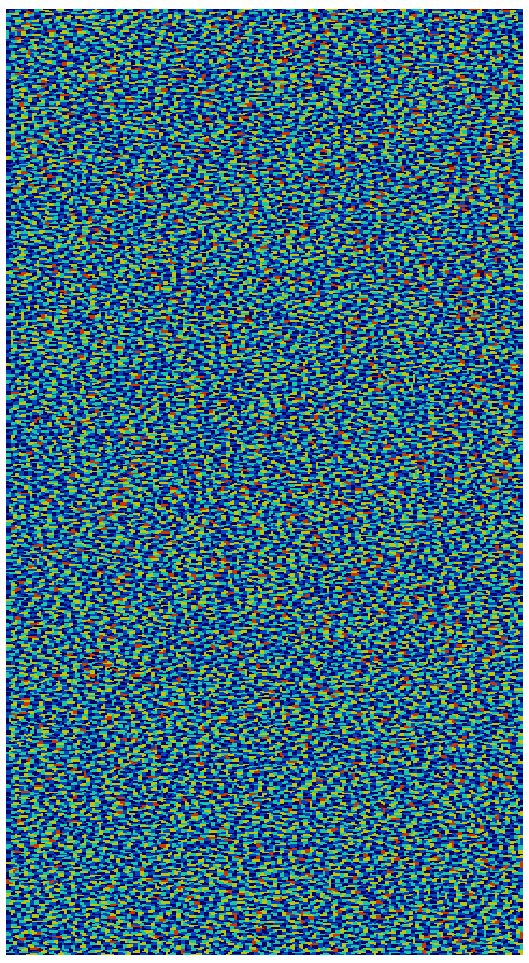}}
         &
         \subfloat[Level $\ell=2$]{\includegraphics[scale=.14]{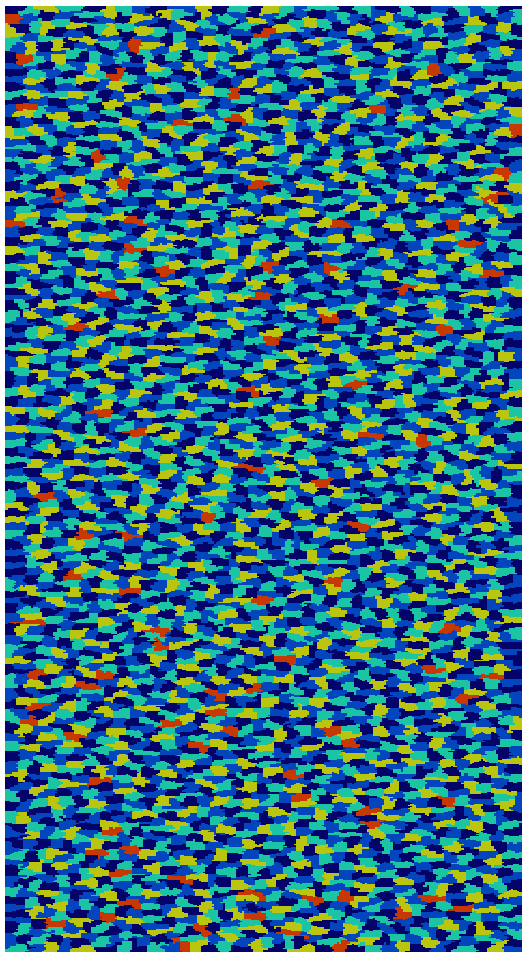}}
         &
         \subfloat[Level $\ell=3$]{\includegraphics[scale=.14]{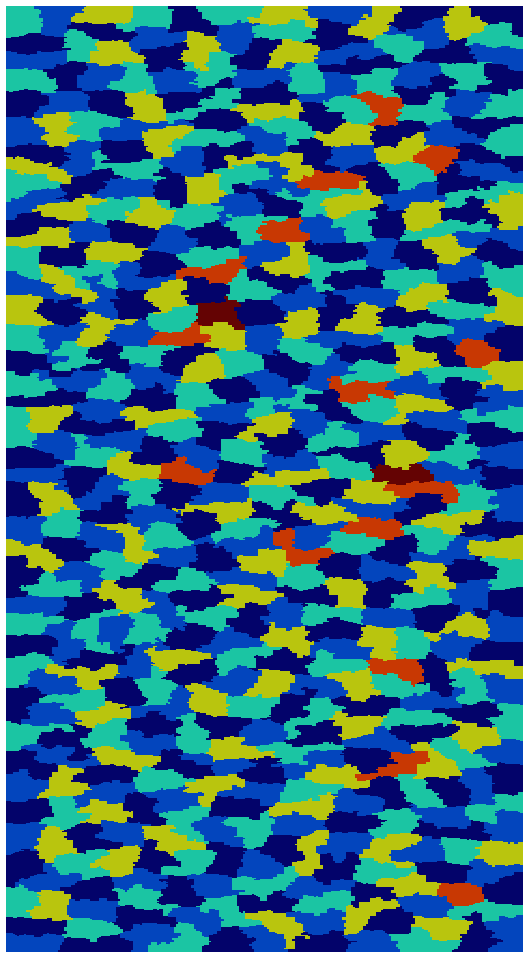}}
     \end{tabular}
     \caption{Nested hierarchy of agglomerated meshes of the domain
     \small{$D=(0,1200)\times(0,2200)$}.}
     \label{fig:2DSPE10MeshHierarchy}
 \end{figure}

 \begin{figure}[htbp]
     \centering
     \begin{tabular}{c c c c c}
         \subfloat[Level $\ell=0$,
         Size of Stochastic Dimension
     $=200K$]{\includegraphics[scale=.18]{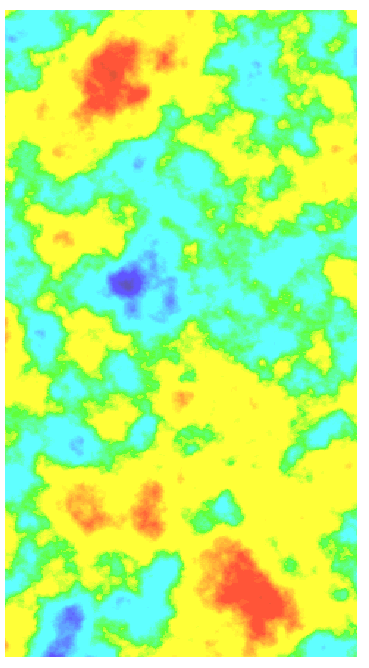}}
         &
         \subfloat[Level $\ell=1$,
         Size of Stochastic Dimension
     $=27K$]{\includegraphics[scale=.18]{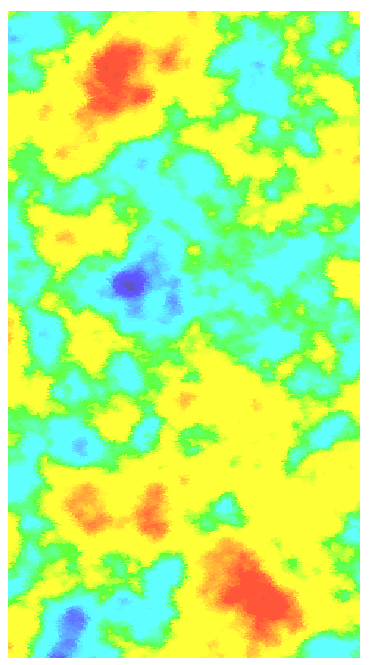}}
         &
         \subfloat[Level $\ell=2$,
         Size of Stochastic Dimension
     $=3.5K$]{\includegraphics[scale=.18]{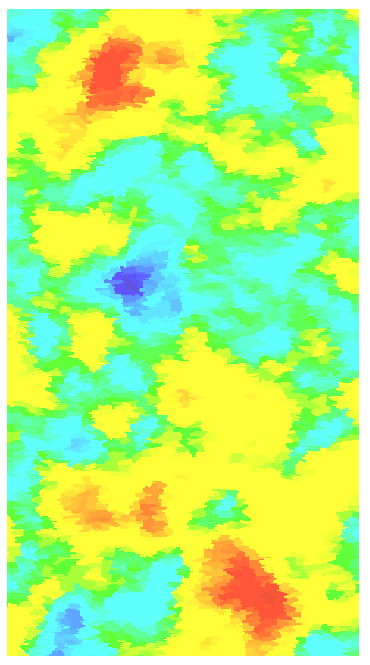}}
         &
         \subfloat[Level $\ell=3$,
         Size of Stochastic Dimension
     $=501$]{\includegraphics[scale=.18]{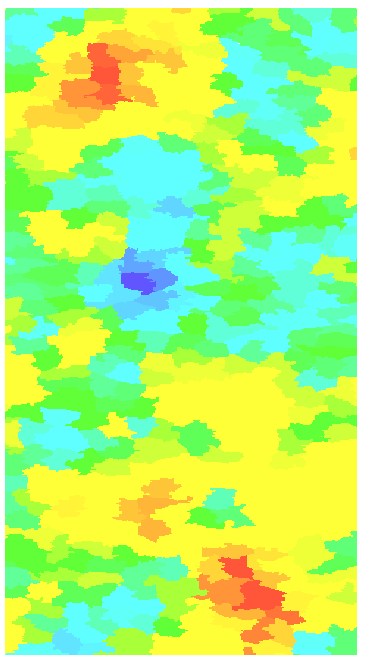}}
        &
         \subfloat{\includegraphics[scale=.15]{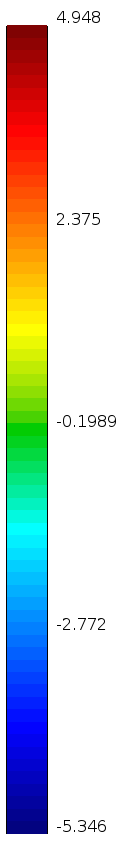}}
     \\\end{tabular}

     \caption{Realizations of Gaussian random fields on the domain $D=(0,1200)\times(0,2200)$ obtained by using our
     hierarchical sampling technique for 4 levels with \Mat covariance with correlation length $b=100$. }
     \label{fig:2DSPE10Samples}
 \end{figure}

 \begin{figure}[htbp]
     \centering
 \includegraphics[scale=.3]{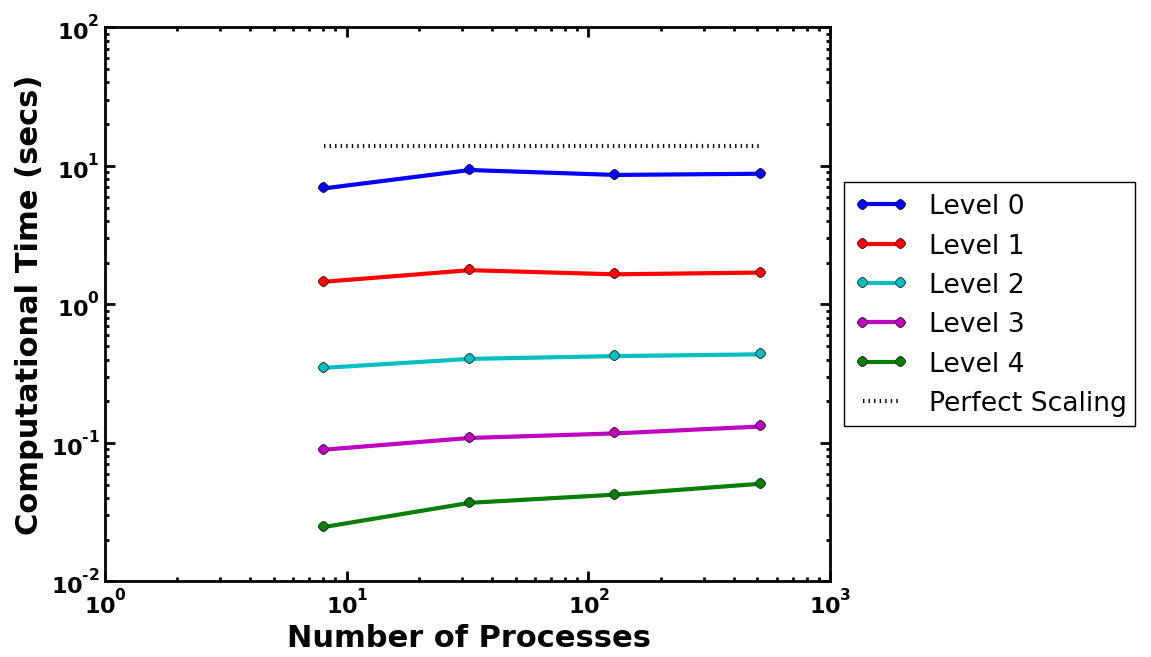}
 \caption{ Weak scalability of the hierarchical SPDE sampler on the domain $D=(0,1200)\times(0,2200)$ where the size of the stochastic dimension per process is fixed. The size of the stochastic dimension of the finest level ranges from 3.7 to 235 million and the number of MPI processes ranges from 8 to 512. The average time to compute a realization using the sampling method exhibits the desired near linear scaling for all levels.}
 \label{fig:2Dscaling}
\end{figure}

As a second example, we consider the generation of large-scale 3D spacially correlated random fields. The computational domain is inspired by one from the SAIGUP benchmark, \cite{saigup}, and represents a realistic geometry of a shallow marine hydrocarbon reservoirs. The computational domain roughly covers an area of $3000 \times 9000$ square meters, and the difference (in depth) between the highest and the lowest point in the domain is $930$ meters.
The original mesh consists of $78720$ hexahedral elements and it is uniformly refined several times to build the hierarchy of levels.
We note that in this example, we have not used embedding of the computational domain into a larger domain, since at this point we are only interested in the scalability of computing a single sample which is done by a scalable multilevel method.

Figure \ref{fig:3DSaigupSamples} shows a sequence of computed Gaussian random field samples $\theta(\x,\omega)$ on different levels (the finer level is obtained by uniformly refining the original mesh twice). To model $\theta(\x,\omega)$ we prescribe the exponential covariance function \eqref{eq:exp_cov_func} with the correlation length $b=100$ meters and unitary marginal variance. Observe the similarity in the fields generated at the different levels.
The weak scalability of the proposed sampling method is demonstrated in Figure \ref{fig:3Dscaling}. The number of MPI processes ranges from 48 to 3072 and the total number of degree of freedom (i.e. the number of unknowns in the mixed system \eqref{eq:linsys}) on the fine grid ranges from 2.1 to 134 million. The proposed sampling method exhibits fairly scalable behavior as the number of processors increases.  

 \begin{figure}[h!]
     \centering
     \begin{tabular}{c c c c}
         \subfloat[Level $\ell=0$,
         Size of Stochastic Dimension = $5M$]{\includegraphics[trim={430 110 400 200}, clip, scale=.22]{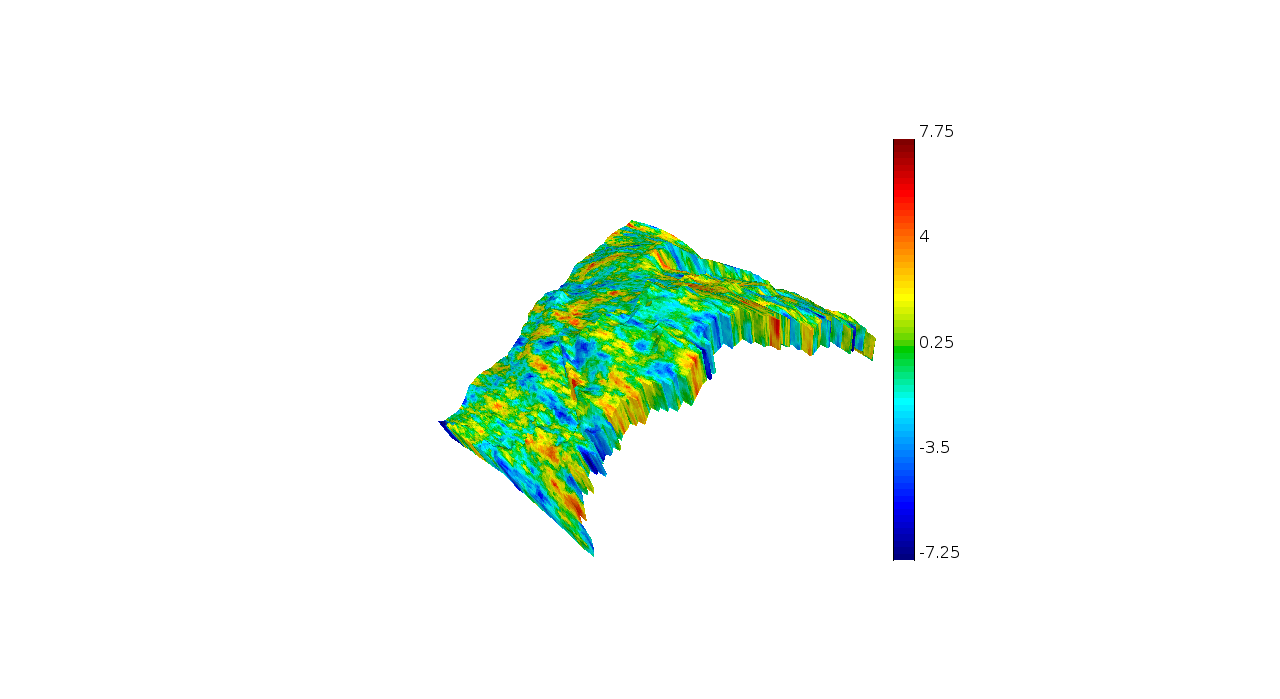}}
         &
         \subfloat[Level $\ell=1$,
         Size of Stochastic Dimension = $630K$]{\includegraphics[trim={430 110 400 200}, clip, scale=.22]{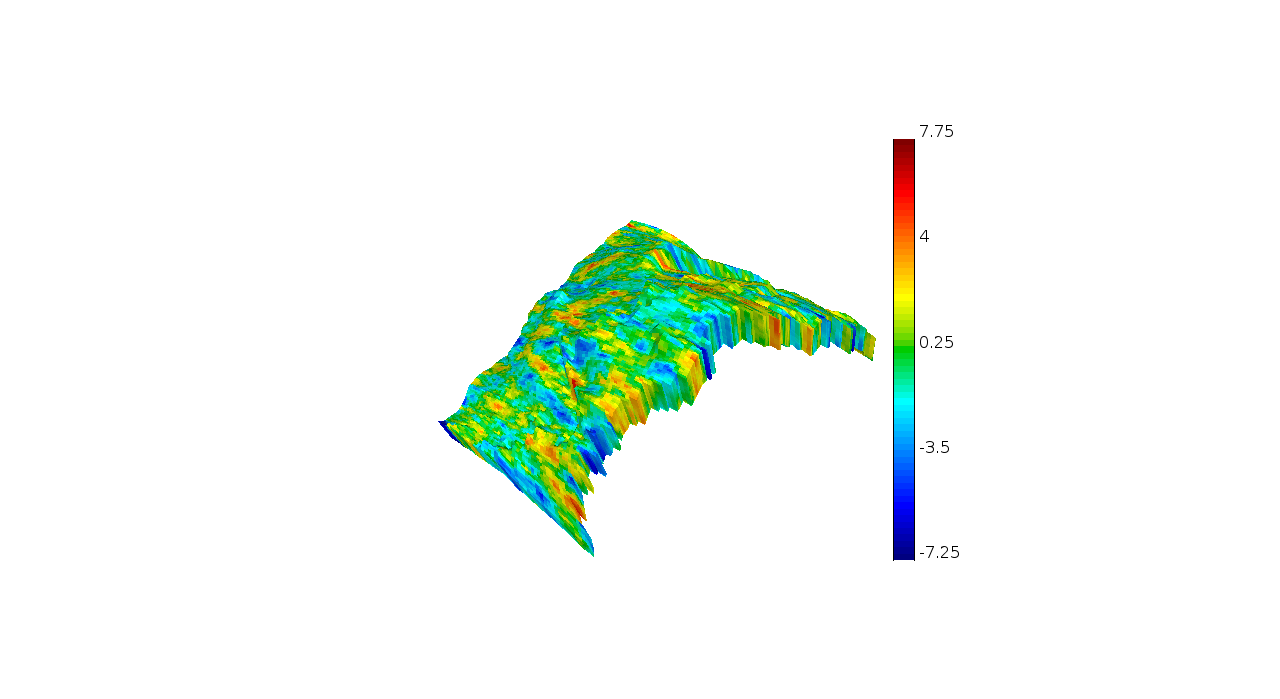}}
         &
         \subfloat[Level $\ell=2$,
         Size of Stochastic Dimension = $78K$]{\includegraphics[trim={430 110 400 200}, clip, scale=.22]{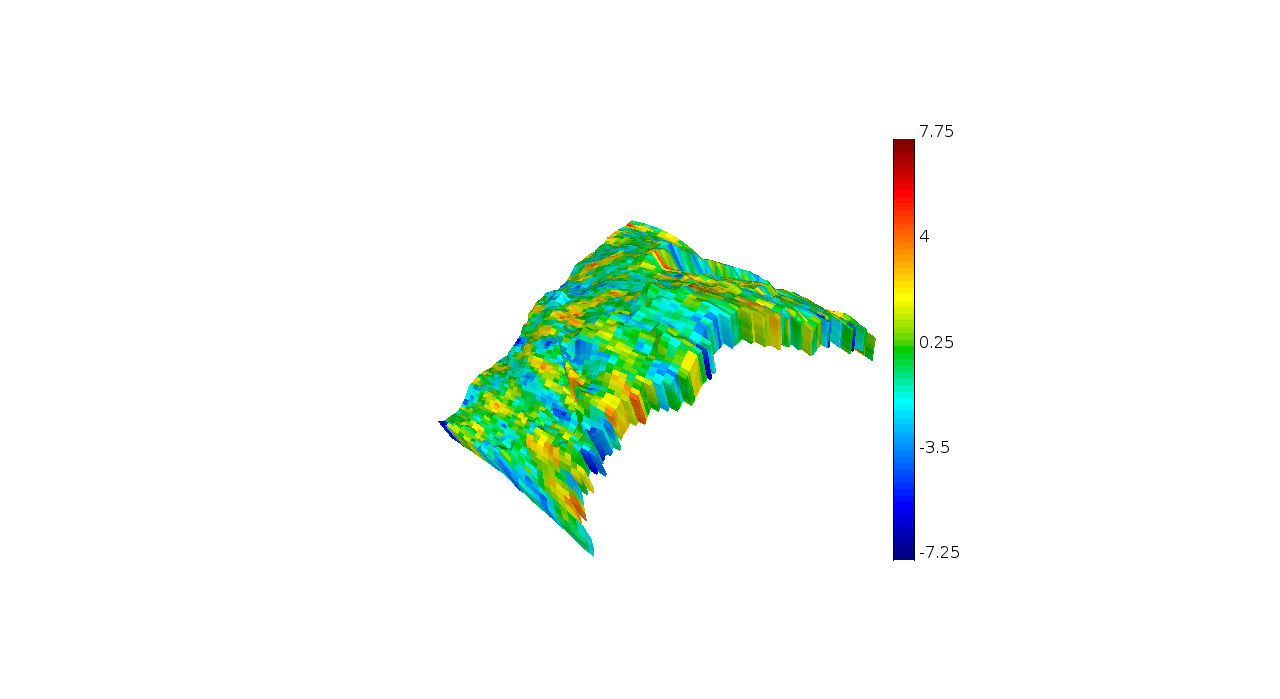}}
         &
         \subfloat{\includegraphics[trim={880 110 300 110}, clip, scale=.20]{Figures/SAIGUP_level2.png}}
 
     \\\end{tabular}
     \caption{Realizations of Gaussian random fields on the SAIGUP domain obtained by using our
     hierarchical sampling technique for 3 levels with exponential covariance function and correlation length $b=100$. }
     \label{fig:3DSaigupSamples}
 \end{figure}

 \begin{figure}[h!]
     \centering
 \includegraphics[width=.75\textwidth]{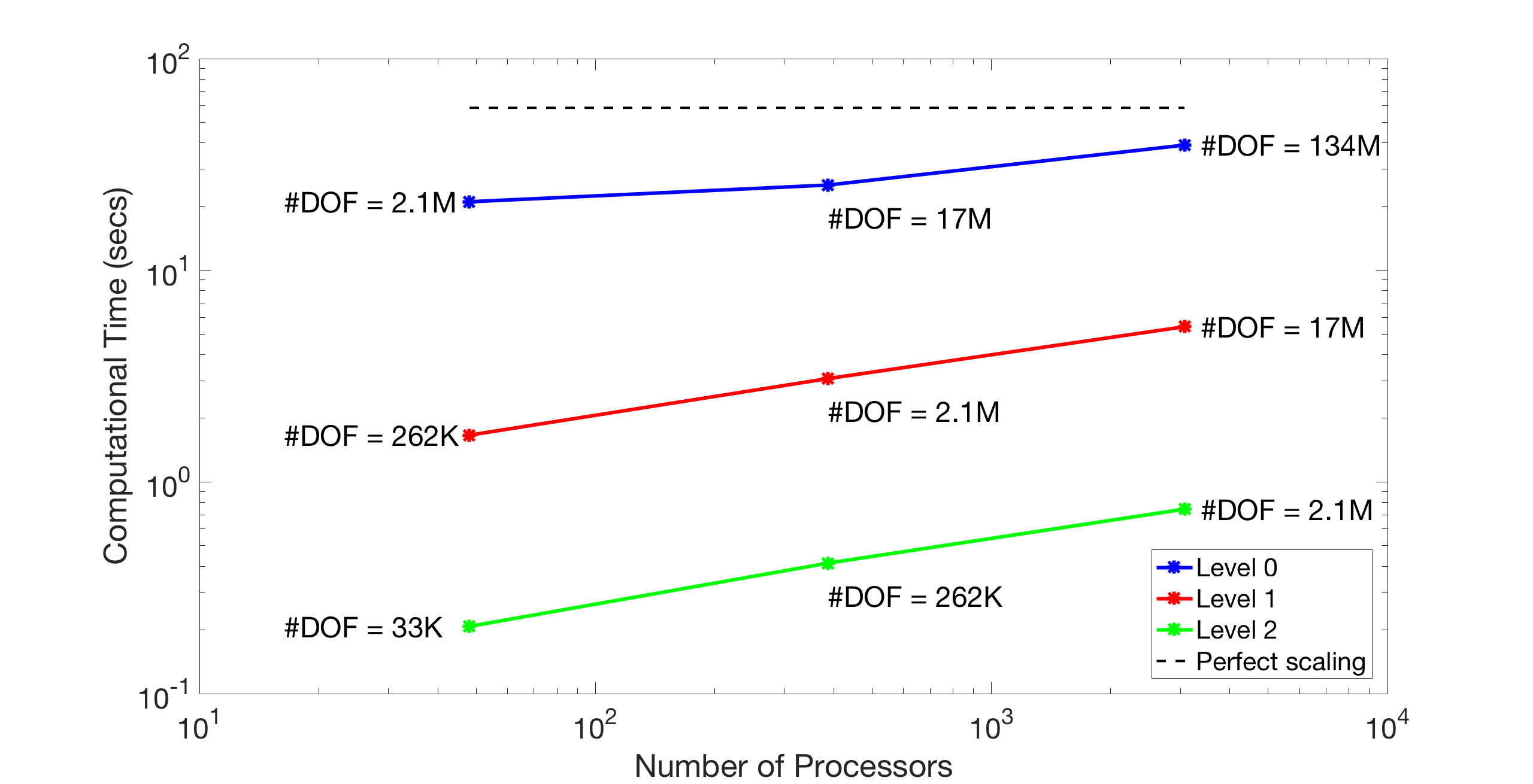}
 \caption{Weak scalability test of the hierarchical SPDE sampler for the 3D SAIGUP domain. The method exhibits good scaling properties for all levels. \#DOF denotes the total number of degree of freedom in the mixed system \eqref{eq:linsys}.}
 \label{fig:3Dscaling}
\end{figure}

\section{Multilevel Monte Carlo Methods}
\label{sec:mlmc}
In this section we briefly describe the Monte Carlo method, and then its
extension to the multilevel Monte Carlo method. Finally, we present results with MLMC simulations for subsurface porous media flow using our hierarchical SPDE sampling method.

In Monte Carlo methods, one is
interested in approximating the expected value of some quantity of interest, $Q(\omega)=\mathcal{B}\mathcal{U}(\x, \omega)$ where $\mathcal{U}({\bf x},\omega)$ is the solution of a PDE with random input coefficients. In our
case, $Q$ represents some functional of the solution $({\bq},p)$ to
\eqref{eq:Darcy1}-\eqref{eq:Darcy2}. In general, the quantity of interest, $Q$,
is inaccessible so an approximation, $Q_{h}$, is computed. The standard Monte
Carlo estimator for the quantity of interest is then
\begin{equation}
    \E[Q] \approx \hat{Q}_{h}^{MC} = \frac{1}{N}\sum_{i=1}^{N}Q_{h}^i,
\end{equation}
where $ Q_{h}^i$ is the $i^{\text{th}}$ sample of $Q_{h}$  
and $N$ is the number of (independent) samples.

The mean square error of the method is given by 
\begin{equation}
    \E[(\hat{Q}_h^{MC} - \E[Q])^2] = \frac{1}{N}\V[Q_{h}]  +
    (\mathbb{E}[Q_h -Q])^{2}. 
\end{equation}
For the root mean square error (RMSE) to be below a prescribed tolerance,
$\varepsilon$, both terms should be smaller than $\varepsilon^{2}/2$. 

The first term is the estimator variance and the second term is the estimator bias. The
estimator bias measures the discretization error and is controlled by the
spatial resolution of the approximate solution.
For sufficiently fine spatial discretizations, 
the estimator bias is small and reducing the mean square error amounts to
reducing the estimator variance.
The estimator variance is then reduced by
increasing the number of samples, $N$. Thus, the standard MC method RMSE converges
in $\bigO(1/\sqrt{N})$. This is favorable as the convergence rate in independent of the stochastic 
dimension of the problem, yet when high accuracy is necessary the number of samples required 
can be a prohibitive expense as the samples must be computed with a fine mesh. For this reason, standard MC methods
are not a scalable approach for the forward propagation of uncertainties when the forward problem is a PDE, due to the
high cost of computing samples with a fine spatial discretization.

This motivates the multilevel Monte Carlo (MLMC) method, see e.g. \cite{cliffe2011multilevel,giles2008multilevel},
which is an effective variance reduction technique based on
hierarchical sampling which aims to alleviate the burden of standard MC by computing samples
on a hierarchy of spatial discretizations.  

Consider the sequence of spatial discretizations 
$\ell=\{0,1,\ldots L\}$, with level $\ell=0$ denoting the finest spatial discretization, 
and level $\ell=L$ the coarsest used to approximate $Q_{\ell}$ where $Q_0 = Q_h$.

The idea behind MLMC is to estimate the correction with respect to the next
coarser level, $\E[Q_{\ell}-Q_{\ell+1}]$ rather than estimating
$\E[Q_{\ell}]$ directly. Linearity of the expectation gives the following
expression for $\E[Q_{h}]$ on the finest level,
\begin{equation}
    \E[Q_{h}] = \E[Q_{L}] +
    \sum_{\ell=0}^{L-1}\mathbb{E}[Q_{\ell}-Q_{\ell+1}] = \sum_{\ell=0}^{L}
    \mathbb{E}[Y_{\ell}], 
\end{equation}
where we have defined $Y_{\ell}= Q_{\ell}-Q_{\ell+1}$ for $i=0,\ldots L-1$ and
$Y_{L}=Q_{L}$.

Similarly, one computes the estimator for $Y_{\ell}$,
\begin{equation}
    \hat{Y}_{\ell}= \frac{1}{N_{\ell}}\sum_{i=1}^{N_{\ell}} \left(
    Q^{(i)}_{\ell}-Q^{(i)}_{\ell+1} \right),
\end{equation}

and then the multilevel Monte Carlo estimator is defined as
\begin{equation}
    \hat{Q}_{h}^{ML} =\sum_{\ell=0}^{L} \hat{Y}_{\ell}.
    \label{eq:MLMCEstimator}
\end{equation}

Consequently, the mean square error for the MLMC method becomes
\begin{equation}
    e(\hat{Q}_{h}^{ML})^{2} = \frac{1}{N_{L}}\mathbb{V}[Q_{L}]
    +\sum_{\ell=0}^{L-1}\frac{1}{N_{\ell}}\mathbb{V}[Y_{\ell}] +
    (\mathbb{E}[Q_{h}-Q])^{2}.
    \label{eq:MLMCMSE}
\end{equation}
The three terms in the right hand side of \eqref{eq:MLMCMSE} represent,
respectively, the variance on the coarsest level, the variance of the
correction with respect to the next coarser level, and lastly the discretization
error. For a prescribed level of accuracy, the number of realizations at the coarsest level, $N_L$, still needs to be large, but samples are much cheaper 
to obtain on the coarser level, and the number of realizations required for levels $(\ell < L)$ given by $N_{\ell}$ is much smaller, since $\V[Q_{\ell}-Q_{\ell+1}] \rightarrow 0$ as $h_{\ell} \rightarrow 0$. Thus, fewer samples are needed for the finest, most computationally expensive level.  
To minimize the overall cost of the MLMC algorithm (e.g. the computational time to reach a desired MSE), the optimal number of samples of each level $N_{\ell}$ is given by
$$ N_{\ell} \propto \sqrt{\frac{\mathbb{V}[Y_{\ell}]}{C_{\ell}}},$$
where $C_{\ell}$ is the cost of computing one sample at level $\ell$.
We refer to \cite{cliffe2011multilevel} for additional details.

\subsection{Numerical Experiments}
In this section we include
standard results from MLMC computations using our proposed SPDE sampler. The forward model is the mixed Darcy equations given by 
\begin{equation}
    \begin{array}{lcr}
        \frac{1}{k(\x,\omega)}\bq(\x,\omega) + \nabla p = 0 & \mbox{ in } D \\
        \nabla \cdot \bq = 0 & \mbox{ in } D,
    \end{array}\label{eq:darcy}
\end{equation}
with homogeneous Neumann boundary conditions $\bq \cdot \bn = 0$ on $\Gamma_N$ and Dirichlet boundary conditions $p = p_D$ on $\Gamma_D$. Here, $\Gamma_N \in \partial D$, $\Gamma_D \in \partial D$ are a non overlapping partition of $\partial D$, and $\bn$ denotes the unit normal vector to $\partial D$. 

We use the mixed finite element method to discretize the model problem \eqref{eq:darcy}, specifically we choose the lowest order Raviart-Thomas element for the flux $\bq$ and piecewise constant functions for the pressure $p$.
Then, for each input realization $k(\x,\omega)$, we write the resulting discretized saddle point problem as 
\begin{equation}\label{eq:mixedDarcyDiscrete}
\mathcal{A}_{k,h}
    \begin{bmatrix}
        \bq_h \\ 
        p_h
    \end{bmatrix}
    =
    \begin{bmatrix}
        \bbf_h \\ 
        0
    \end{bmatrix},
\quad \text{where }
\mathcal{A}_{k,h} =
\begin{bmatrix}
        M_{k,h} & B_h^T \\
        B_h & 0
 \end{bmatrix}.
\end{equation}
Here $\bbf_h$ stems from the discretization of the Dirichlet boundary condition $p = p_D$ on $\Gamma_D$.
The large sparse indefinite linear system \eqref{eq:mixedDarcyDiscrete} is solved using MINRES preconditioned with a block diagonal preconditioner based on the pressure Schur Complement, namely the $L^2-H^1$ preconditioner described in \cite{MardalWinther2011}. Specifically, we consider the symmetric positive definite preconditioner
\begin{equation*}
\mathcal{N}_{k,h} = 
\begin{bmatrix}
H_{k,h} & 0 \\
0 & \widetilde{\Sigma}_{k,h}
\end{bmatrix},
\end{equation*}
where $H_{k,h} = \text{diag}(M_{k,h})$ and $\widetilde{\Sigma}_{k,h} = B_h H^{-1}_{k,h} B_h^T$.

It is well-known that $\mathcal{N}_{k,h}$ is a robust preconditioner for \eqref{eq:mixedDarcyDiscrete} as long as $M_{k,h}$ is not too anisotropic. In fact, as shown in \cite{murphy2000note}, $\mathcal{N}_{k,h}$ is an optimal preconditioner for $\widetilde{\mathcal{A}}_{k,h}$,
\begin{equation*}
\widetilde{\mathcal{A}}_{k,h} = 
    \begin{bmatrix}
        H_{k,h} & B_h^T \\
        B_h & 0
\end{bmatrix},
\end{equation*}
and $\widetilde{\mathcal{A}}_{k,h}$ is spectrally equivalent to $\mathcal{A}_{k,h}$, since the Raviart-Thomas finite element matrix $M_{k,h}$ is spectrally equivalent to its diagonal $H_{k,h}$.

In the computations, we use BoomerAMG from hypre~\cite{hypre} to precondition the Schur complement $\widetilde{\Sigma}_{k,h}$ which is explicitly available and sparse.
It should be noted that the AMG preconditioner of $\widetilde{\Sigma}_{k,h}$ is recomputed for each input realization.
For the simulations, an absolute stopping criteria of $10^{-12}$ and a relative stopping 
criteria of $10^{-6}$ is used for the linear solver.  

\subsubsection{Top Layer of SPE10 Dataset}
First we show experiments incorporating data from the Tenth SPE Benchmark (SPE10)~\cite{spe_data}. We consider a 2D slice of the dataset of dimension $1200 \times 2200\, {\rm ft}^2$ divided into cells of size $20 \times 10\,  {\rm ft}^2$ resulting in a mesh with $60 \times 220$ quadrilateral elements. 
The PDE coefficient on each slice is a scalar function. The original $60 \times 220$ quadrilateral mesh corresponds to the coarsest one in our MLMC experiments.
To produce the other (finer) levels, we uniformly refine the initial 2D mesh several times.

We have $D = (0,1200) \times (0,2200)$ and assume the random conductivity coefficient $k(\x,\omega)$ is modeled as a log-normal random field.  
A realization of $k(\x,\omega)$ is generated by computing the exponential of a realization of a Gaussian random field. In particular, we assume that the mean of the Gaussian random field is the logarithm of the top horizontal slice from the SPE10 dataset. 

Figure \ref{subfig:kappa} shows a particular realization of the random conductivity coefficient $k(\x,\omega)$ modeled as a log-normal field where 
$k(\x,\omega) = \operatorname{exp}\brac{{\operatorname {log}[k_{SPE10slice}(\x)] + \theta(\omega)}} $ where $\theta(\omega)$ is a realization of the Gaussian random field generated with our sampler shown in Figure \ref{subfig:real} and $k_{SPE10slice}(\x)$ is shown in Figure \ref{subfig:kinv}.

 \begin{figure}[htbp]
     \begin{tabular}{c c c}
         \subfloat[ Top layer of SPE10]{\includegraphics[scale=.28]{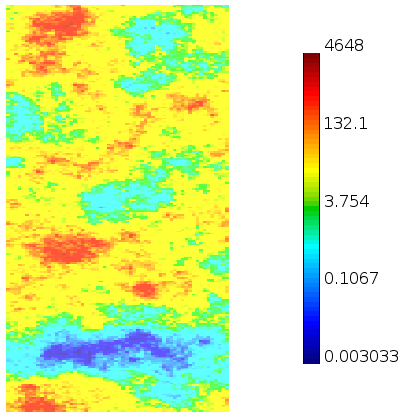}\label{subfig:kinv}} 
         &
     \subfloat[Realization of Gaussian random field $\theta(\omega)$]{\includegraphics[scale=.28]{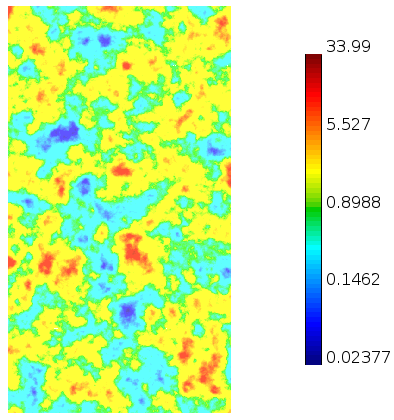}\label{subfig:real}}
         & 
     \subfloat[Realization of conductivity coefficient $k(\x,\omega)$ ]{\includegraphics[scale=.28]{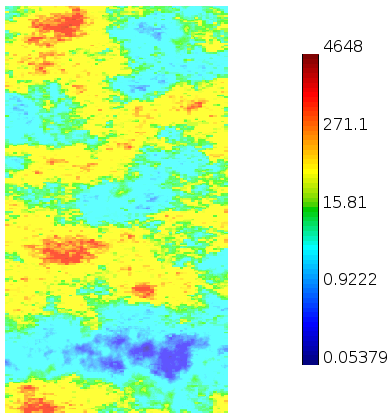}\label{subfig:kappa}}
         
     \\\end{tabular}
     \caption{A log-normal realization of the random conductivity coefficient $k(\x,\omega)$ shown in (c) is computed as the product of the top layer of the SPE10 dataset and the exponential of a realization of a Gaussian random field obtained by using our
     hierarchical sampling technique (logarithmic scale).}
     \label{fig:SPE10_real}
 \end{figure}

We solve \eqref{eq:darcy} with the following boundary conditions:

    \begin{equation*}
        \begin{cases}
            -p&=1 \quad \text{on} \ \Gamma_{in} := (0,1200) \times \curly{0}\\
            -p&=0 \quad \text{on} \ \Gamma_{out} := (0,1200) \times \curly{2200} \\
            \bq \cdot {\bf n} &= 0\quad  \text{on} \ \Gamma_{s} := \partial D \setminus \left(\Gamma_{in} \cup \Gamma_{out}\right). \\
        \end{cases}
    \end{equation*}

The quantity of interest we use is the expected value of the
effective permeability, that is the (horizontal) flux through the ``outflow''
part of the boundary, defined as
\begin{equation}
    k_{eff}(\omega) = \frac{1}{\abs{\Gamma_{out}}} \int_{\Gamma_{out}} \bq 
    (\cdot,\omega)\cdot{\bf n} \, dS.
\label{eq:qoi}
\end{equation}

Figure \ref{fig:SPE10_ml} contains four subplots relating to the multilevel estimator and performance of the multilevel Monte Carlo method with hierarchical, SPDE sampling. The target mean square error is $\epsilon^2 = 8.73e$-5 for Figures \ref{subfig:SPE10_mean}-\ref{subfig:SPE_solve}. The first figure, Figure \ref{subfig:SPE10_mean},
displays the multilevel estimator where the blue line with circles represents the
expectation at each level $\mathbb{E}[Q_{\ell}]$ and the green dashed line represents the expectation of the difference in
levels, $\mathbb{E}[Q_{\ell}-Q_{\ell+1}]$.
Figure \ref{subfig:SPE10_var} illustrates the multilevel variance reduction. This plot
contains two lines, the blue line represents the variance of the particular level,
whereas the green dashed line represents the variance of the difference in levels. 
The plot shows the effectiveness of the MLMC method at reducing
the variance as the number of unknowns increases. 
The average sampling time to generate the required Gaussian field realizations and solve the forward model for each level is shown in Figure
\ref{subfig:SPE_solve}. 
This plot indicates near optimal scaling of the MLMC method
with the proposed hierarchical sampler.
Figure \ref{subfig:DofsvSamps} shows the number of samples required at each
level of the MLMC method for different prescribed
mean square error tolerances. The plot clearly shows that more samples are
generated on the coarse levels (fewer degrees of freedom) than on the finest
levels (many degrees of freedom). This merely confirms the MLMC theory with our
proposed hierarchical sampling technique.

\begin{figure}[htbp]
\begin{center}
\begin{minipage}{.49\textwidth}
    \subfloat[Multilevel
    estimator]{\includegraphics[scale=.3]{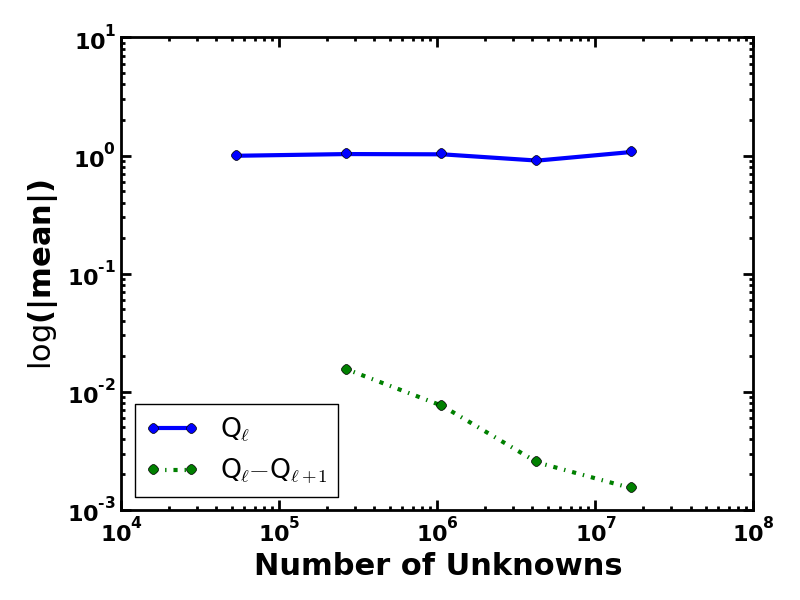}\label{subfig:SPE10_mean}}
\end{minipage}
\begin{minipage}{.49\textwidth}
    \subfloat[Multilevel variance
    reduction]{\includegraphics[scale=.3]{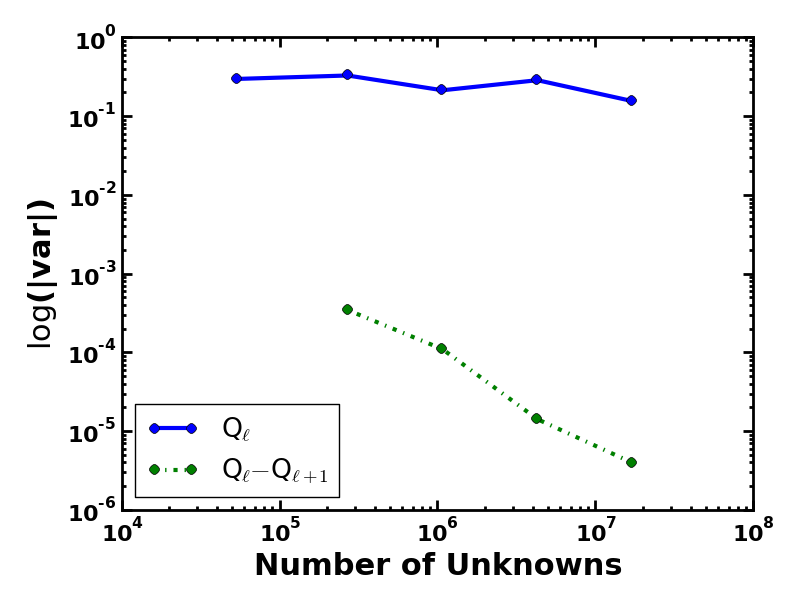}\label{subfig:SPE10_var}}
\end{minipage}
\begin{minipage}{.49\textwidth}
        \subfloat[Average sample time versus number of
        unknowns]{\includegraphics[scale=.3]{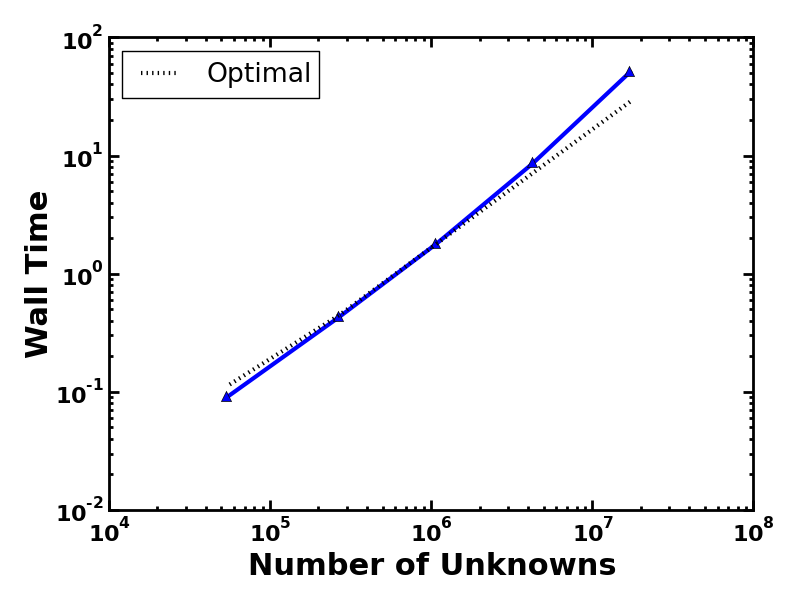}\label{subfig:SPE_solve}}
\end{minipage}
\begin{minipage}{.49\textwidth}
    \subfloat[Number of samples for each level for varying MSE
        tolerances]{\includegraphics[scale=.3]{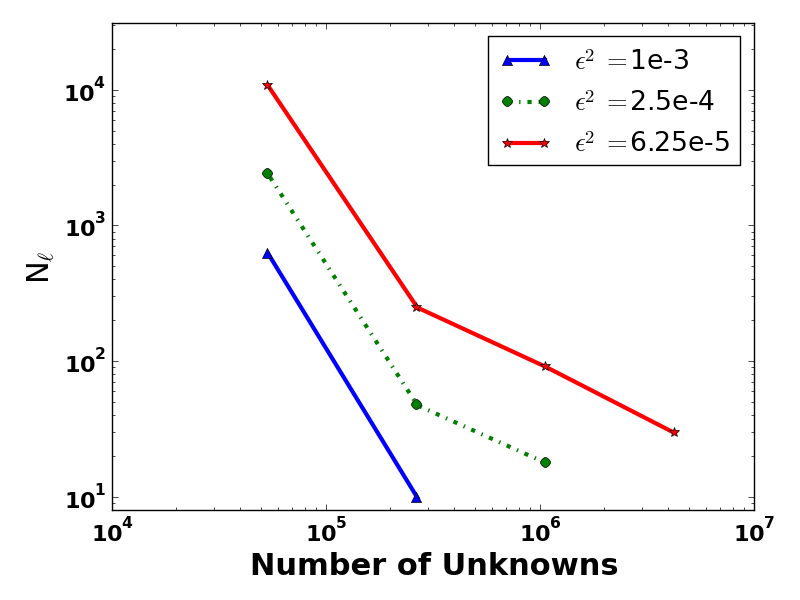}\label{subfig:DofsvSamps}}
    \end{minipage}

    \end{center}
    \caption{Using the SPE10 data with the computational domain $D = (0,1200) \times (0,2200)$, the effective permeability is estimated using MLMC. The plots (a)-(c) show the MLMC estimator, variance reduction, and average sampling time to generate the required Gaussian field realizations and solve the forward model for each level versus the number of unknowns where the target mean square error is $\epsilon=8.72e$-5. The number of samples $N_{\ell}$ per levels is shown in (d) for varying target MSE accuracy levels.}
\label{fig:SPE10_ml}
\end{figure}

\subsubsection{Unit Cube}
\label{sec:num_cube}
In this section we present similar experiments for estimating the expectation of the effective permeability, but
for the unit cube domain $D=(0,1)\times (0,1) \times (0,1)$. 

We solve \eqref{eq:darcy} with the following boundary conditions:

    \begin{equation*}
        \begin{cases}
            -p&=1 \quad \text{on} \ \Gamma_{in} := (0,1)\times(0,1)\times \curly{0}\\
            -p&=0 \quad \text{on} \ \Gamma_{out} := (0,1)\times(0,1)\times \curly{1} \\
            \bq \cdot {\bf n} &= 0\quad  \text{on} \ \Gamma_{s} := \partial D \setminus \left(\Gamma_{in} \cup \Gamma_{out}\right). \\
        \end{cases}
    \end{equation*}
The quantity of interest is the expected value of the
effective permeability defined in \eqref{eq:qoi}.
The original mesh consists of 64 hexahedral elements and is uniformly refined several times to build the hierarchy of levels.
We examine the performance of the multilevel estimator for the unit cube in Figure \ref{fig:cube_ml} which contains four subplots. 
The first figure, Figure \ref{subfig:cube_mean},
displays the multilevel estimator where the blue line with circles represents the
expectation at each level $\mathbb{E}[Q_{\ell}]$ and the green dashed line represents the expectation of the difference in
levels, $\mathbb{E}[Q_{\ell}-Q_{\ell+1}]$.

Figure \ref{subfig:cube_var} illustrates the multilevel variance reduction of the method, while 
Figure \ref{subfig:cube_solve} shows the average sampling time required to generate the Gaussian random field realizations and solve the forward model problem for each level. The method with hierarchical, SPDE sampling exhibits near optimal scaling for the 3D problem formulation.
The number of samples required at each
level of the MLMC method for different prescribed
mean square error tolerances is shown in \ref{subfig:cube_nsamples}. 
\begin{figure}[htbp]
    \begin{center}
\begin{minipage}{.49\textwidth}
    \subfloat[Multilevel
    estimator]{\includegraphics[scale=.3]{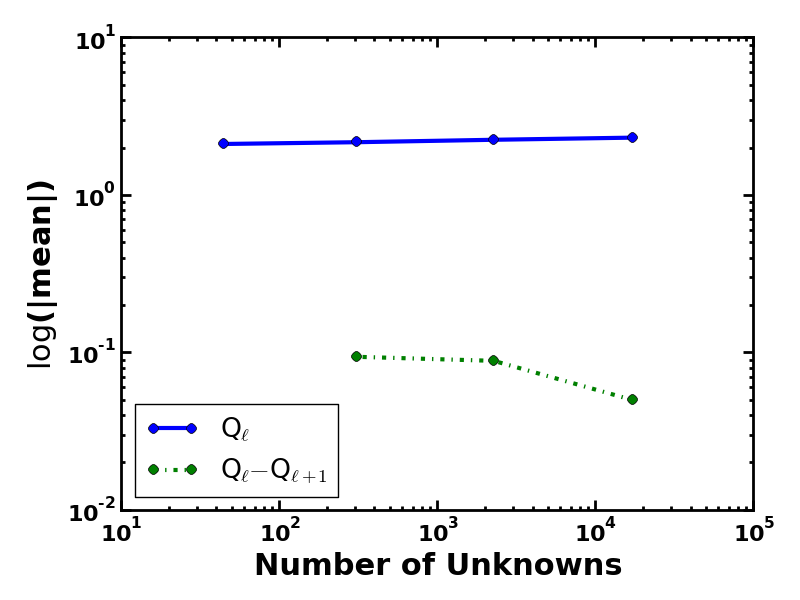}\label{subfig:cube_mean}}
\end{minipage}
\begin{minipage}{.49\textwidth}
    \subfloat[Multilevel variance
    reduction]{\includegraphics[scale=.3]{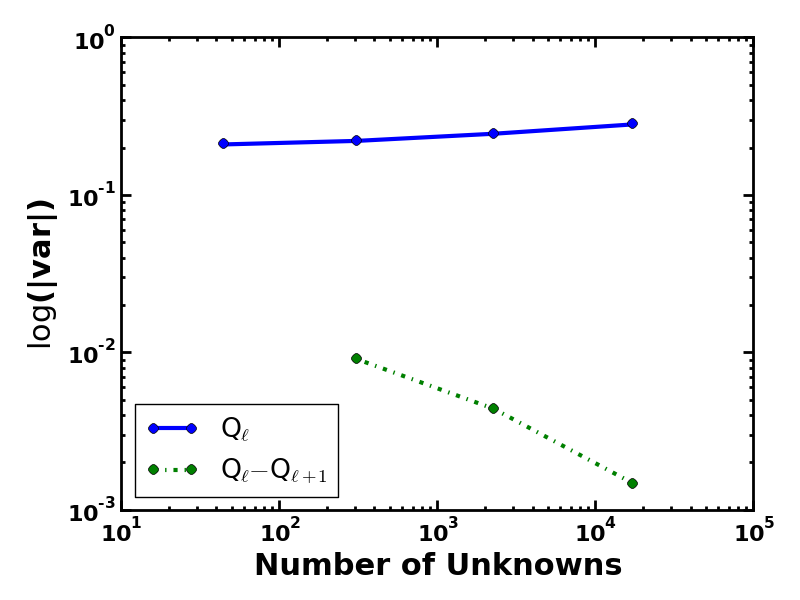}\label{subfig:cube_var}}
\end{minipage}
\begin{minipage}{.49\textwidth}
        \subfloat[Average sample time versus number of
        unknowns]{\includegraphics[scale=.3]{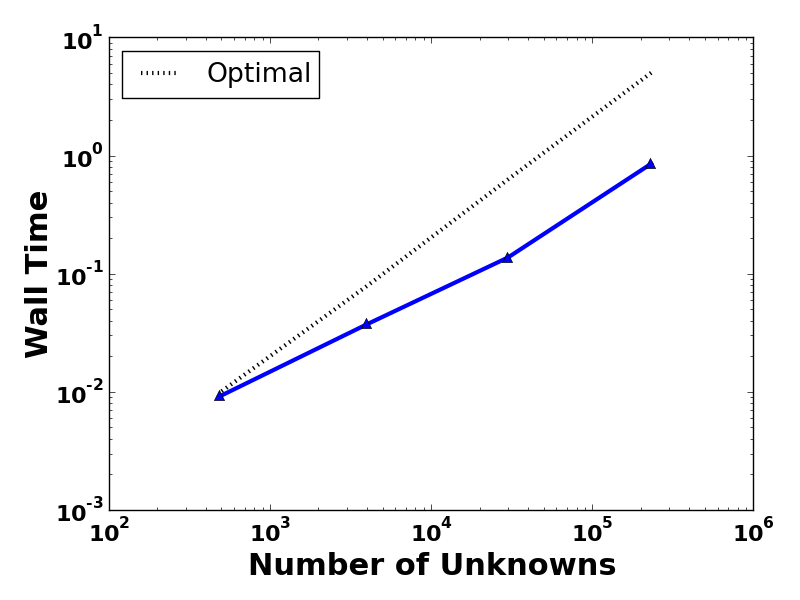}\label{subfig:cube_solve}}
\end{minipage}
\begin{minipage}{.49\textwidth}
        \subfloat[Number of samples for each level for varying MSE tolerance.]{\includegraphics[scale=.3]{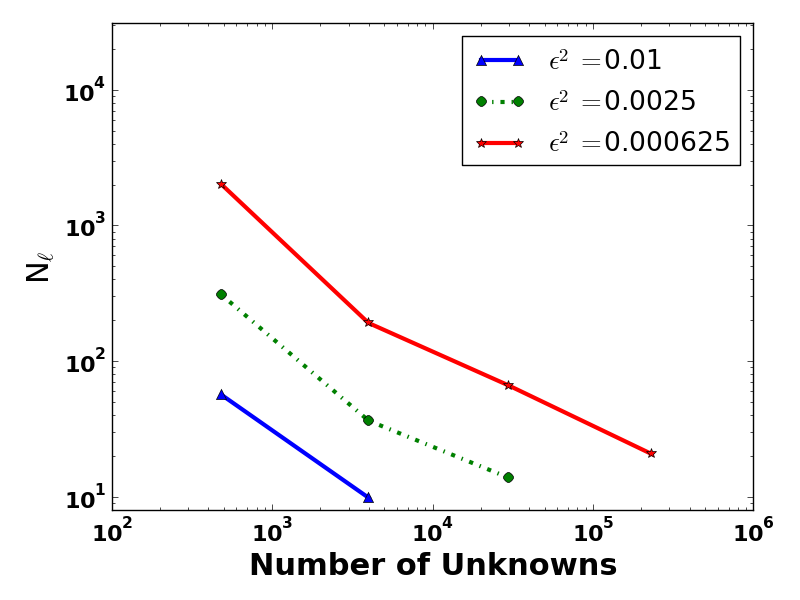}\label{subfig:cube_nsamples}}
\end{minipage}
\end{center}
\caption{The effective permeability is estimated using MLMC with the unit cube as the computational domain. The MLMC estimator, variance reduction, and average sampling time to generate the Gaussian random field realizations and solve the forward model problem for each level for the target mean square error of $\epsilon=6.25e$-5 are shown in plots (a)-(c) and the number of samples for each level of the MLMC method for varying target MSE is shown in plot (d).}
\label{fig:cube_ml}
\end{figure}

For both of the examined computational domains, the hierarchical, SPDE sampler yields the expected results for MLMC variance reduction and displays the desired scaling properties for the possibility of large-scale MLMC simulations.
 
\section{Conclusions}
\label{sec:conclusions}
Multilevel Monte Carlos simulations for PDEs with uncertain input coefficients employ a hierarchy of spatial resolutions as a variance reduction technique for the approximation of expected quantities of interest.
A key component in the multilevel Monte Carlo method is the ability to generate samples of a random field at different spatial resolutions.
The \KL expansion provides a parametrization independent of the spatial discretization. However, both the computation and the memory requirements become infeasible at large-scale as the expansion requires the ability to compute and store eigenpairs of a large, dense covariance matrix.
We suggest a sampling method based on the solution of a particular stochastic PDE.
This method is highly scalable, but the parametrization is mesh dependent.
We have proposed a multilevel decomposition of the stochastic field to allow for scalable, hierarchical stochastic PDEs samplers. Numerical results are provided that suggest the method possesses the desired scalability as the method leverages existing scalable solvers. We also have applied the new sampling technique to MLMC simulations of subsurface flow problems with over 10 million parameters in the stochastic dimension.

\bibliographystyle{siamplain}

\end{document}